\documentclass{amsart}

\usepackage[all]{xy}
\xyoption{poly}
\xyoption{arc}
\xyoption{curve}

\CompileMatrices

\begin{document}

\newtheorem{theorem}{Theorem}{}
\newtheorem{lemma}[theorem]{Lemma}{}
\newtheorem{corollary}[theorem]{Corollary}{}
\newtheorem{conjecture}[theorem]{Conjecture}{}
\newtheorem{proposition}[theorem]{Proposition}{}
\newtheorem{axiom}{Axiom}{}
\newtheorem{remark}{Remark}{}
\newtheorem{example}{Example}{}
\newtheorem{exercise}{Exercise}{}
\newtheorem{definition}{Definition}{}

\newcommand{\nc}{\newcommand}

\nc{\gr}{\bullet}
\nc{\pr}{\noindent{\em Proof. }} \nc{\g}{\mathfrak g}
\nc{\n}{\mathfrak n} \nc{\opn}{\overline{\n}}\nc{\h}{\mathfrak h}
\renewcommand{\b}{\mathfrak b}
\nc{\Ug}{U(\g)} \nc{\Uh}{U(\h)} \nc{\Un}{U(\n)}
\nc{\Uopn}{U(\opn)}\nc{\Ub}{U(\b)} \nc{\p}{\mathfrak p}
\renewcommand{\l}{\mathfrak l}
\nc{\z}{\mathfrak z} \renewcommand{\h}{\mathfrak h}
\nc{\m}{\mathfrak m}
\renewcommand{\t}{\mathfrak t}
\renewcommand{\k}{\mathfrak k}
\nc{\opk}{\overline{\k}}
\nc{\opb}{\overline{\b}}

\renewcommand{\thetheorem}{\thesection.\arabic{theorem}}

\renewcommand{\thelemma}{\thesection.\arabic{lemma}}

\renewcommand{\theproposition}{\thesection.\arabic{proposition}}

\renewcommand{\thecorollary}{\thesection.\arabic{corollary}}

\renewcommand{\theremark}{\thesection.\arabic{remark}}

\renewcommand{\thedefinition}{\thesection.\arabic{definition}}

\title{The structure of q-W algebras}

\author{A. Sevostyanov}
\address{ Institute of Pure and Applied Mathematics,
University of Aberdeen \\ Aberdeen AB24 3UE, United Kingdom }

\begin{abstract}
We suggest two explicit descriptions of the Poisson q-W algebras which are Poisson algebras of
regular functions on certain algebraic group analogues of the Slodowy transversal
slices to adjoint orbits in a complex semisimple Lie algebra $\g$. To obtain the first description we introduce
certain projection operators which are analogous to the quasi--classical versions of the so--called Zhelobenko and extremal projection operators. As a byproduct we obtain some new formulas for natural coordinates on Bruhat cells in algebraic groups.
\end{abstract}


\maketitle

\pagestyle{myheadings}
\markboth{A. SEVOSTYANOV}{}

\renewcommand{\theequation}{\thesection.\arabic{equation}}


\section{Introduction}

\setcounter{equation}{0}
\setcounter{theorem}{0}

The purpose of this paper is to give two descriptions of the Poisson structure for Poisson q-W algebras which are group--like analogues of ordinary Poisson W-algebras. As algebras Poisson q-W algebras are algebras of regular functions on transversal slices to conjugacy slices in algebraic groups defined in \cite{S6} while their Poisson structures are obtained by Poisson reduction from Poisson structures of Poisson--Lie groups dual to quasi--triangular Poisson Lie groups. If $G$ is a complex semi--simple quasi--triangular Poisson--Lie group used in the definition of q-W algebras then its dual group $G^*$ is solvable and there is a smooth morphism of manifolds $q:G^*\rightarrow G$ the image of which is dense in $G$. Actually this morphism is a finite cover over its image, so that one can equip $G$ with a Poisson structure such that $q$ becomes a Poisson mapping. Denote by $G_*$ the Poisson manifold obtained this way. In order to define the corresponding Poisson q-W--algebra one has to fix a coisotropic submanifold $\mathcal{C}\subset G_*$ and then factorize it over a free action of a subgroup $N\subset G$ on $\mathcal{C}$ induced by the conjugation action of $N$ on $G$. The reduced Poisson manifold $\Sigma=\mathcal{C}/N$ has an explicit description as a submanifold in $\mathcal{C}$ and carries the reduced Poisson structure.

In fact $G$ and $N$ are closed algebraic groups and $\mathcal{C}\subset G_*$ is a closed subvariety, so that $\Sigma$ is closed subvariety in $G_*$, and the Poisson q-W algebra is the Poisson algebra of regular functions on $\Sigma$. Our first task is to obtain a projection operator $\Pi:\mathbb{C}[\mathcal{C}]\rightarrow \mathbb{C}[\mathcal{C}]^N=\mathbb{C}[\Sigma]$ with the help of which one can explicitly describe the reduced Poisson structure in terms of the Poisson bracket on $G_*$ (see Proposition \ref{pred2}).

The projection operator $\Pi$ is quite remarkable. It is an analogue of the quasi--classical versions of the so--called Zhelobenko and extremal projection operators introduced in \cite{Z2} to describe subspaces of singular vectors in modules from the BGG category $\mathcal{O}$. In turn, as it was observed in \cite{SZ}, the quasi--classical versions of Zhelobenko and extremal projection operators are examples of the realization of the following simple construction.

Let $U$ be an algebraic group, and $M$ an algebraic variety equipped with a regular action of $U$. Assume that there exists a cross--section $X\subset M$ for this action in the sense that the map
\begin{equation}\label{iso}
U\times X \rightarrow M,~(g,x)\mapsto gx,~g\in U,x\in X
\end{equation}
is an isomorphism of varieties.

Denote by $\mathbb{C}[M]$ the algebra of regular functions on $M$. The group $U$ naturally acts on $\mathbb{C}[M]$. Let $\mathbb{C}[M]^U$ be the subspace of $U$--invariant elements of $\mathbb{C}[M]$.
Define the projection operator $\overline{P}:\mathbb{C}[M]\rightarrow \mathbb{C}[M]^U$ as follows
\begin{equation}\label{P}
(\overline{P}f)(gx)=f(x)=f(g^{-1}gx)=(gf)(gx),~g\in U,x\in X.
\end{equation}
If for $y\in M$ we denote by $g(y)\in U$ the unique element such that $y=g(y)x(y)$ for a unique $x(y)\in X$ then
\begin{equation}\label{P1}
(\overline{P}f)(y)=(g(y)f)(y).
\end{equation}
The operator $\overline{P}$ is called the projection operator corresponding to isomorphism (\ref{iso}).

If every element of $U$ can be uniquely represented as a product of elements from subgroups $U_1, \ldots, U_n$, i.e. $U=U_1\cdot \ldots \cdot U_n$, then the operator $\overline{P}$ can be expressed as a composition of operators $\overline{P}_i$,
\begin{equation}\label{Pi}
(\overline{P}_if)(y)=(g_i(y)f)(y),~g(y)=g_1(y)\ldots g_n(y),~g_i(y)\in U_i,
\end{equation}
\begin{equation}\label{PC}
(\overline{P}f)(y)=(\overline{P}_1\ldots \overline{P}_nf)(y).
\end{equation}

The projection operator $\Pi:\mathbb{C}[\mathcal{C}]\rightarrow \mathbb{C}[\mathcal{C}]^N=\mathbb{C}[\Sigma]$ defined in Proposition \ref{proj} is of the same type with $U=N$, $M=\mathcal{C}$ and $U_i$ being one--parametric subgroups corresponding to some roots. Technically the construction is quite involved and as a preliminary axillary exercise we obtain some new formulas for natural coordinates on Bruhat cells in algebraic groups (see Proposition \ref{Br}).

In Proposition \ref{pred1} we give another explicit compact description of the reduced Poisson bracket on $\Sigma$. Formula (\ref{tau1}) for the reduced Poisson structure obtained in that proposition is a simple consequence of the general definition of the reduced Poisson structure and of the explicit description of the reduced space $\Sigma$. In case of the usual W-algebras such description was obtained in \cite{KDV}.  But this description is much more complicated than in case of q-W algebras since the corresponding reduced spaces in the W-algebra case are the Slodowy slices, and they do not have a nice geometric description compatible with the Poisson structure.


\section{Analogues of Zhelobenko operators for q-W algebras}
\label{Zhel}

\setcounter{equation}{0}
\setcounter{theorem}{0}

In this section we construct analogues of the classical versions of Zhelobenko operators for q-W--algebras. First we recall the definition of the coisotropic submanifolds $\mathcal{C}$, of the subgroups $N$ and of the slices $\Sigma$. They are associated to conjugacy classes of Weyl group elements and are defined in terms of certain systems of positive roots associated to conjugacy classes of Weyl group elements. We recall these definitions and the related properties of $\mathcal{C}$, $N$ and $\Sigma$ following \cite{S6,S10,S11,S12}. To define the projection operator $\Pi$ we shall also need the normal orderings of the systems of positive roots associated to conjugacy classes of Weyl group elements in \cite{S10}.

Let $G$ be a complex semisimple connected simply--connected algebraic group, $\g$ its Lie
algebra, $H$ a maximal torus in $G$. Denote by $\h$ the Cartan subalgebra in $\g$ corresponding to $H$. Let $\Delta$ be the root system of the pair $(\g,\h)$, $W$ the Weyl group of the pair $(\g,\h)$. For any root $\alpha\in \Delta$ we denote by $\alpha^\vee\in \h$ the corresponding coroot and by $s_\alpha\in W$ the reflection with respect to $\alpha$. Let $X_\alpha\in \g$ be a non--zero root vector corresponding to $\alpha$.

Let $s$ be an element of the Weyl group $W$ and $\h_{\mathbb{R}}$ the real form of $\h$, the real linear span of simple coroots in $\h$. The set of roots $\Delta$ is a subset of the dual space $\h_\mathbb{R}^*$.
By Theorem C in \cite{C} $s$ can be represented as a product of two involutions,
\begin{equation}\label{inv}
s=s^1s^2,
\end{equation}
where $s^1=s_{\gamma_1}\ldots s_{\gamma_n}$, $s^2=s_{\gamma_{n+1}}\ldots s_{\gamma_{l'}}$, the roots in each of the sets $\gamma_1, \ldots, \gamma_n$ and ${\gamma_{n+1}}, \ldots, {\gamma_{l'}}$ are positive and mutually orthogonal, and
the roots $\gamma_1, \ldots, \gamma_{l'}$ form a linear basis of $\h'^*$.

The Weyl group element $s$ naturally acts on $\h_{\mathbb{R}}$ as an orthogonal transformation with respect to the scalar product induced by the Killing form of $\g$. Using the spectral theory of orthogonal transformations we can decompose $\h_{\mathbb{R}}$ into a direct orthogonal sum of $s$--invariant subspaces,
\begin{equation}\label{hdec}
\h_\mathbb{R}=\bigoplus_{i=0}^{K} \h_i,
\end{equation}
where we assume that $\h_0$ is the linear subspace of $\h_{\mathbb{R}}$ fixed by the action of $s$, and each of the other subspaces $\h_i\subset \h_\mathbb{R}$, $i=1,\ldots, K$, is either two--dimensional or one--dimensional and the Weyl group element $s$ acts on it as rotation with angle $\theta_i$, $0<\theta_i<\pi$ or as reflection with respect to the origin, respectively. Note that since $s$ has finite order $\theta_i=\frac{2\pi n_i}{m_i}$, $n_i,m_i\in \{1,2,3,\ldots \}$. By Proposition 3.1 in \cite{S11} the subspaces $\h_i$ can be chosen in such a way that each of them is invariant with respect to the involutions $s^1$ and $s^2$, and if $\h_i$ is one--dimensional one of the involutions acts on it in the trivial way. We shall assume that the subspaces $\h_i$ are chosen in this way. We shall also assume that the one--dimensional subspaces $\h_i$ on which $s^1$ acts by multiplication by $-1$ are immediately preceding $\h_0$ in sum (\ref{hdec}).

Since the number of roots in the root system $\Delta$ is finite one can always choose elements $h_i\in \h_i$, $i=0,\ldots, K$, such that $h_i(\alpha)\neq 0$ for any root $\alpha \in \Delta$ which is not orthogonal to the $s$--invariant subspace $\h_i$ with respect to the natural pairing between $\h_{\mathbb{R}}$ and $\h_{\mathbb{R}}^*$.

Now we consider certain $s$--invariant subsets of roots ${\Delta}_i$, $i=0,\ldots, K$, defined as follows
\begin{equation}\label{di}
{{\Delta}}_i=\{ \alpha\in \Delta: h_j(\alpha)=0, j>i,~h_i(\alpha)\neq 0 \},
\end{equation}
where we formally assume that $h_{K+1}=0$.
Note that for some indexes $i$ the subsets ${{\Delta}}_i$ are empty, and that the definition of these subsets depends on the order of terms in direct sum (\ref{hdec}).

 Now consider the nonempty $s$--invariant subsets of roots ${\Delta}_{i_k}$, $k=0,\ldots, M$.
For convenience we assume that indexes $i_k$ are labeled in such a way that $i_j<i_k$ if and only if $j<k$.
According to this definition ${\Delta}_{0}=\{\alpha \in \Delta: s\alpha=\alpha\}$ is the set of roots fixed by the action of $s$. Observe also that the root system $\Delta$ is the disjoint union of the subsets ${\Delta}_{i_k}$,
$$
\Delta=\bigcup_{k=0}^{M}{\Delta}_{i_k}.
$$

Suppose that the direct sum $\bigoplus_{k=1}^r\h_{i_k}$ of  the one--dimensional subspaces $\h_{i_k}$ on which $s^1$ acts by multiplication by $-1$ and $s^2$ acts as the identity transformation is not trivial.
Since  the one--dimensional subspaces $\h_i$ on which $s^1$ acts by multiplication by $-1$ are immediately preceding $\h_0$ in sum (\ref{hdec}), the roots from the union $\bigcup_{k=1}^r{\Delta}_{i_k}$ must be orthogonal to all subspaces $\h_{i_k}$ on which $s^1$ does not act by multiplication by $-1$ and to all roots from the set $\gamma_{n+1}, \ldots \gamma_{l'}$ as $s^2$ acts trivially on $\bigoplus_{k=1}^r\h_{i_k}$. Pick up a root $\gamma \in \bigcup_{k=1}^r{\Delta}_{i_k}$. Then $s^1\gamma=-\gamma$ and by our choice $\gamma$ is orthogonal to the roots $\gamma_{n+1}, \ldots \gamma_{l'}$. Therefore $s^1_0=s^1s_\gamma$ is an involution the dimension of the fixed point space of which is equal to the dimension of the fixed point space of the involution $s^1$ plus one, and $s^2_0=s_\gamma s^2$ is another involution the dimension of the fixed point space of which is equal to the dimension of the fixed point space of the involution $s^1$ minus one. We also have a decomposition $s=s^1_0s^2_0$.

Now we can apply the above construction of the system of positive roots to the new decomposition of $s$. Iterating this procedure we shall eventually arrive to the situation when the direct sum $\bigoplus_{k=1}^r\h_{i_k}$ of  the subspaces $\h_{i_k}$ on which $s^1$ acts by multiplication by $-1$ is trivial. Form now on we shall only consider decompositions $s=s^1s^2$ which satisfy this property. This implies
\begin{equation}\label{cond2}
s^2\alpha=\alpha \Rightarrow \alpha \in \Delta_0.
\end{equation}

Now assume that
\begin{equation}\label{cond}
|h_{i_k}(\alpha)|>|\sum_{l\leq j<k}h_{i_j}(\alpha)|, ~{\rm for~any}~\alpha\in {\Delta}_{i_k},~k=0,\ldots, M,~l<k.
\end{equation}
Condition (\ref{cond}) can be always fulfilled by suitable rescalings of the elements $h_{i_k}$.

Consider the element
\begin{equation}\label{hwb}
\bar{h}=\sum_{k=0}^{M}h_{i_k}\in \h_\mathbb{R}.
\end{equation}

From definition (\ref{di}) of the sets ${\Delta}_i$ we obtain that for $\alpha \in {\Delta}_{i_k}$
\begin{equation}\label{dech}
\bar{h}(\alpha)=\sum_{j\leq k}h_{i_j}(\alpha)=h_{i_k}(\alpha)+\sum_{j< k}h_{i_j}(\alpha)
\end{equation}
Now condition (\ref{cond}), the previous identity and the inequality $|x+y|\geq ||x|-|y||$ imply that for $\alpha \in {\Delta}_{i_k}$ we have
$$
|\bar{h}(\alpha)|\geq ||h_{i_k}(\alpha)|-|\sum_{j< k}h_{i_j}(\alpha)||>0.
$$
Since $\Delta$ is the disjoint union of the subsets ${\Delta}_{i_k}$, $\Delta=\bigcup_{k=0}^{M}{\Delta}_{i_k}$, the last inequality ensures that  $\bar{h}$ belongs to a Weyl chamber of the root system $\Delta$, and one can define the subset of positive roots $\Delta_+^s$ with respect to that chamber. We call $\Delta_+^s$ a system of positive roots associated to $s$.

Recall that an ordering of a set of positive roots $\Delta_+$ is called normal if for any three roots $\alpha,~\beta,~\gamma$ such that $\gamma=\alpha+\beta$ we have either $\alpha<\gamma<\beta$ or $\beta<\gamma<\alpha$. Let $\alpha_1,\ldots, \alpha_l$ be the simple roots in $\Delta_+$, $s_1, \ldots ,s_l$ the corresponding simple reflections. Let $\overline{w}$ be the element of $W$ of maximal length with respect to the system $s_1, \ldots ,s_l$ of simple reflections. For any reduced decomposition $\overline{w}=s_{i_1}\ldots s_{i_D}$ of $\overline{w}$ the ordering
$$
\beta_1=\alpha_{i_1},\beta_2=s_{i_1}\alpha_{i_2},\ldots,\beta_D=s_{i_1}\ldots s_{i_{D-1}}\alpha_{i_D}
$$
is a normal ordering in $\Delta_+$, and there is a one--to--one correspondence between normal orderings of $\Delta_+$ and reduced decompositions of $\overline{w}$ (see \cite{Z1}).

From this fact and from properties of Coxeter groups it follows that any two normal orderings in $\Delta_+$ can be reduced to each other by the so--called elementary transpositions (see \cite{Z1}, Theorem 1). The elementary transpositions for rank 2 root systems are inversions of the following normal orderings (or the inverse normal orderings):
\begin{equation}\label{rank2}
\begin{array}{lr}
\alpha,~\beta & A_1+A_1 \\
\\
\alpha,~\alpha+\beta,~\beta & A_2 \\
\\
\alpha,~\alpha+\beta,~\alpha+2\beta,~\beta & B_2 \\
\\
\alpha,~\alpha+\beta,~2\alpha+3\beta,~\alpha+2\beta,~\alpha+3\beta,~\beta & G_2
\end{array}
\end{equation}
where it is assumed that $(\alpha,\alpha)\geq (\beta,\beta)$. Moreover, any normal ordering in a rank 2 root system is one of orderings (\ref{rank2}) or one of the inverse orderings.

In general an elementary inversion of a normal ordering in a set of positive roots $\Delta_+$  is the inversion of an ordered segment of form (\ref{rank2}) (or of a segment with the inverse ordering) in the ordered set $\Delta_+$, where $\alpha-\beta\not\in \Delta$.

Let $\beta_{1}, \beta_{2}, \ldots, \beta_{D}$ be a normal ordering
of a positive root system $\Delta_+$. Then following \cite{KT3} one can introduce the corresponding circular normal ordering of the root system ${\Delta}$ where
the roots in ${\Delta}$ are located on a circle in
the following way
\begin{center}
\setlength{\unitlength}{0.6mm}
     \begin{picture}(180,120)(-40,0)
     \put(0,50){\makebox(0,0){$\beta_1$}}
     \put(4,69){\makebox(0,0){$\beta_2$}}
     \put(15,85){\circle*{1.5}}
     \put(31,96){\circle*{1.5}}
     \put(50,100){\circle*{1.5}}
     \put(69,96){\circle*{1.5}}
     \put(85,85){\circle*{1.5}}
     \put(96,69){\makebox(0,0){$\beta_D$}}
     \put(100,50){\makebox(0,0){$-\beta_1$}}
     \put(96,31){\makebox(0,0){$-\beta_2$}}
     \put(85,15){\circle*{1.5}}
     \put(69,4){\circle*{1.5}}
     \put(50,0){\circle*{1.5}}
     \put(31,4){\circle*{1.5}}
     \put(15,15){\circle*{1.5}}
     \put(4,31) {\makebox(0,0){-$\beta_D$}}
     \put(64,88){\vector (3,-2){10}}
     \put(36,12){\vector (-3,2){10}}
     \end{picture}
\end{center}
\begin{center}
 Fig.2
\end{center}

Let $\alpha,\beta\in \Delta$. One says that the segment $[\alpha, \beta]$ of the circle
is minimal if it does not contain the opposite roots $-\alpha$ and $-\beta$ and the root $\beta$ follows after $\alpha$ on the circle above, the circle being oriented clockwise.
In that case one also says that $\alpha < \beta$ in the sense of the circular normal ordering,
\begin{equation}\label{noc}
\alpha < \beta \Leftrightarrow {\rm the ~segment}~ [\alpha, \beta]~{\rm  of~ the ~circle~
is~ minimal}.
\end{equation}

Later we shall need the following property of minimal segments which is a direct consequence of Proposition 3.3 in \cite{kh-t}.
\begin{lemma}\label{minsegm}
Let $[\alpha, \beta]$ be a minimal segment in a circular normal ordering of a root system $\Delta$. Then if $\alpha+\beta$ is a root we have
$$
\alpha<\alpha+\beta<\beta.
$$
\end{lemma}

The following proposition is a refinement of Proposition 5.1 in \cite{S10}.
\begin{proposition}\label{pord}
Let $s\in W$ be an element of the Weyl group $W$ of the pair $(\g,\h)$, $\Delta$ the root system of the pair $(\g,\h)$, $\Delta_+^s$ a system of positive roots associated to $s$. For any Weyl group element $w\in W$ we denote $\Delta_w=\{\alpha\in \Delta_+^s:w\alpha\in \Delta_-^s\}$, where $\Delta_-^s=-\Delta_+^s$. Then
the decomposition $s=s^1s^2$ is reduced in the sense that ${l}(s)={l}(s^2)+{l}(s^1)$, where ${l}(\cdot)$ is the length function in $W$ with respect to the system of simple roots in $\Delta_+^s$, and $\Delta_{s}=\Delta_{s^{2}}\bigcup s^2(\Delta_{s^{1}})$, $\Delta_{s^{-1}}=\Delta_{s^{1}}\bigcup s^1(\Delta_{s^{2}})$ (disjoint unions). Here $s^1,s^2$ are the involutions entering decomposition (\ref{inv}), $s^1=s_{\gamma_1}\ldots s_{\gamma_n}$, $s^2=s_{\gamma_{n+1}}\ldots s_{\gamma_{l'}}$ and satisfying (\ref{cond2}), the roots in each of the sets $\gamma_1, \ldots, \gamma_n$ and ${\gamma_{n+1}},\ldots, {\gamma_{l'}}$ are positive and mutually orthogonal.

Moreover, there is a normal ordering of the root system $\Delta_+^s$ of the following form
\begin{eqnarray}
\beta_1^1,\ldots, \beta_t^1,\beta_{t+1}^1, \ldots,\beta_{t+\frac{p-n}{2}}^1, \gamma_1,\beta_{t+\frac{p-n}{2}+2}^1, \ldots , \beta_{t+\frac{p-n}{2}+n_1}^1, \gamma_2, \nonumber \\
\beta_{t+\frac{p-n}{2}+n_1+2}^1 \ldots , \beta_{t+\frac{p-n}{2}+n_2}^1, \gamma_3,\ldots, \gamma_n, \beta_{t+p+1}^1,\ldots, \beta_{l(s^1)}^1,\ldots, \label{NO} \\
\beta_1^2,\ldots, \beta_q^2, \gamma_{n+1},\beta_{q+2}^2, \ldots , \beta_{q+m_1}^2, \gamma_{n+2}, \beta_{q+m_1+2}^2,\ldots , \beta_{q+m_2}^2, \gamma_{n+3},\ldots,  \nonumber \\
\gamma_{l'},\beta_{q+m_{l(s^2)}+1}^2, \ldots,\beta_{2q+2m_{l(s^2)}-(l'-n)}^2, \beta_{2q+2m_{l(s^2)}-(l'-n)+1}^2,\ldots, \beta_{l(s^2)}^2, \nonumber \\
\beta_1^0, \ldots, \beta_{D_0}^0, \nonumber
\end{eqnarray}
where
\begin{eqnarray*}
\{\beta_1^1,\ldots, \beta_t^1,\beta_{t+1}^1, \ldots,\beta_{t+\frac{p-n}{2}}^1, \gamma_1,\beta_{t+\frac{p-n}{2}+2}^1, \ldots , \beta_{t+\frac{p-n}{2}+n_1}^1, \gamma_2, \nonumber \\
\beta_{t+\frac{p-n}{2}+n_1+2}^1 \ldots , \beta_{t+\frac{p-n}{2}+n_2}^1, \gamma_3,\ldots, \gamma_n, \beta_{t+p+1}^1,\ldots, \beta_{l(s^1)}^1\}=\Delta_{s^1},
\end{eqnarray*}
\begin{eqnarray*}
\{\beta_{t+1}^1, \ldots,\beta_{t+\frac{p-n}{2}}^1, \gamma_1,\beta_{t+\frac{p-n}{2}+2}^1, \ldots , \beta_{t+\frac{p-n}{2}+n_1}^1, \gamma_2, \nonumber \\
\beta_{t+\frac{p-n}{2}+n_1+2}^1 \ldots , \beta_{t+\frac{p-n}{2}+n_2}^1, \gamma_3,\ldots, \gamma_n\}=\{\alpha\in \Delta_+^s|s^1(\alpha)=-\alpha\},
\end{eqnarray*}
\begin{eqnarray*}
\{\beta_1^2,\ldots, \beta_q^2, \gamma_{n+1},\beta_{q+2}^2, \ldots , \beta_{q+m_1}^2, \gamma_{n+2}, \beta_{q+m_1+2}^2,\ldots , \beta_{q+m_2}^2, \gamma_{n+3},\ldots,  \nonumber \\
\gamma_{l'},\beta_{q+m_{l(s^2)}+1}^2, \ldots,\beta_{2q+2m_{l(s^2)}-(l'-n)}^2, \beta_{2q+2m_{l(s^2)}-(l'-n)+1}^2,\ldots, \beta_{l(s^2)}^2\}=\Delta_{s^2},
\end{eqnarray*}
\begin{eqnarray*}
\{\gamma_{n+1},\beta_{q+2}^2, \ldots , \beta_{q+m_1}^2, \gamma_{n+2}, \beta_{q+m_1+2}^2,\ldots , \beta_{q+m_2}^2, \gamma_{n+3},\ldots,  \nonumber \\
\gamma_{l'},\beta_{q+m_{l(s^2)}+1}^2, \ldots,\beta_{2q+2m_{l(s^2)}-(l'-n)}^2\}=\{\alpha\in \Delta_+^s|s^2(\alpha)=-\alpha\},
\end{eqnarray*}
\begin{equation*}
\{\beta_1^0, \ldots, \beta_{D_0}^0\}=\{\alpha\in \Delta_+^s|s(\alpha)=\alpha\}.
\end{equation*}

The length of the ordered segment $\Delta_{\m_+}\subset \Delta$ in normal ordering (\ref{NO}),
\begin{eqnarray}
\Delta_{\m_+}=\gamma_1,\beta_{t+\frac{p-n}{2}+2}^1, \ldots , \beta_{t+\frac{p-n}{2}+n_1}^1, \gamma_2, \beta_{t+\frac{p-n}{2}+n_1+2}^1 \ldots , \beta_{t+\frac{p-n}{2}+n_2}^1, \nonumber \\
\gamma_3,\ldots, \gamma_n, \beta_{t+p+1}^1,\ldots, \beta_{l(s^1)}^1,\ldots, \beta_1^2,\ldots, \beta_q^2, \label{dn} \\
\gamma_{n+1},\beta_{q+2}^2, \ldots , \beta_{q+m_1}^2, \gamma_{n+2}, \beta_{q+m_1+2}^2,\ldots , \beta_{q+m_2}^2, \gamma_{n+3},\ldots, \gamma_{l'}, \nonumber
\end{eqnarray}
is equal to
\begin{equation}\label{dimm}
D-(\frac{l(s)-l'}{2}+D_0),
\end{equation}
where $D$ is the number of roots in $\Delta_+^s$, $l(s)$ is the length of $s$ and $D_0$ is the number of positive roots fixed by the action of $s$.

For any two roots $\alpha, \beta\in \Delta_{\m_+}$ such that $\alpha<\beta$ the sum $\alpha+\beta$ cannot be represented as a linear combination $\sum_{k=1}^qc_k\gamma_{i_k}$, where $c_k\in \{0,1,2,\ldots \}$ and $\alpha<\gamma_{i_1}<\ldots <\gamma_{i_k}<\beta$.

The roots from the set $\Delta_s$ form a minimal segment in $\Delta_+^s$ of the form $\gamma, \ldots , \beta_{l(s^2)}^2$ which contains $\Delta_{s^2}$.

Normal ordering (\ref{NO}) also satisfies the property that for any $\alpha\in ({\Delta}_{i_k})_+$ such that $s\alpha \in ({\Delta}_{i_k})_+$ one has $s\alpha>\alpha$, and if $\beta, \gamma\in {\Delta}_{i_{j}}$, $j<k$ and $s\alpha+\beta, \alpha+\gamma\in \Delta$ then $s\alpha+\beta,\alpha+\gamma\in \Delta_+^s$ and $s\alpha+\beta>\alpha+\gamma$.

In particular, for any $\alpha \in \Delta_+^s$, $\alpha\not\in \Delta_0$ and any $\alpha_0\in \Delta_0$ such that $s\alpha \in \Delta_+^s$ one has $s\alpha>\alpha$ and if $s\alpha +\alpha_0\in \Delta$ then $s\alpha +\alpha_0>\alpha$.
\end{proposition}

\begin{proof}
The proof is a refinement of the proof of Proposition 5.1 in \cite{S10} where a normal ordering associated to $s$ and satisfying all properties listed in this proposition except for the last three is constructed.

We describe first the set $({\Delta}_{i_k})_+={\Delta}_{i_k}\bigcap \Delta_+^s$. Suppose that the corresponding $s$--invariant subspace $\h_{i_k}$ is a two--dimensional plane. The case when $\h_{i_k}$ is an invariant line on which $s^2$ acts by reflection and $s^1$ acts trivially can be treated in a similar way. The plane $\h_{i_k}$ is shown at Figure 1.

$$
\xy/r10pc/: ="A",-(1,0)="B", "A",+(1,0)="C","A",-(0,1)="D","A",+(0.1,0.57)*{h_{i_k}},"A", {\ar+(0,+0.5)},"A", {\ar@{-}+(0,+1)}, "A";"B"**@{-},"A";"C"**@{-},"A";"D"**@{-},"A", {\ar+(0.6,0.13)},"A",+(0.65,0.18)*{v^2_k},"A",+(0.87,0.18)*{{\Delta}_{i_k}^{2}},"A", {\ar@{-}+(0.9,0.41)}, "A",+(0.32,0)="D", +(-0.038,0.134)="E","D";"E" **\crv{(1.33,0.07)},"A",+(0.45,0.15)*{\psi_k},"A",+(0.49,0.05)*{\psi_k},"A", {\ar+(-0.6,0.23)},"A",+(-0.65,0.28)*{v^1_k},"A",+(-0.87,0.28)*{{\Delta}_{i_k}^{1}},"A", {\ar@{-}+(-0.85,0.72)}, "A",+(-0.32,0)="F", +(0.066,0.21)="G","F";"G" **\crv{(0.67,0.07)},"A",+(-0.45,0.27)*{\varphi_k},"A",+(-0.49,0.08)*{\varphi_k},"A", {\ar@{--}+(0.4,0.9)},"A",+(0.87,0.69)*{s^2{\Delta}_{i_k}^{1}}
\endxy
$$
\begin{center}
 Fig.1
\end{center}

The vector $h_{i_k}$ is directed upwards at the picture. By (\ref{cond}) and  (\ref{dech}) a root $\alpha\in {\Delta}_{i_k}$ belongs to the set $({\Delta}_{i_k})_+$ if and only if $h_{i_k}(\alpha)>0$. Identifying $\h_\mathbb{R}$ and $\h_\mathbb{R}^*$ with the help of the Killing form one can deduce that $\alpha\in {\Delta}_{i_k}$ is in $\Delta_+^s$ iff its orthogonal, with respect to the Killing form, projection onto $\h_{i_k}$ is contained in the upper--half plane shown at Figure 1.

The element $s$ acts on $\h_{i_k}$ by clockwise rotation with the angle $\theta_{i_k}=2(\varphi_k+\psi_k)$. Therefore the set $\Delta_s\bigcap {\Delta}_{i_k}$ consists of the roots the orthogonal projections of which onto $\h_{i_k}$ belong to the union of the sectors labeled $s^2{\Delta}_{i_k}^{1}$ and ${\Delta}_{i_k}^{2}$ at Figure 1.

Define other $s$--invariant subsets of roots $\overline{\Delta}_{i_k}$, $k=0,\ldots, M$,
\begin{equation}\label{dik}
\overline{\Delta}_{i_k}=\bigcup_{i_j\leq i_k}{\Delta}_{i_j}.
\end{equation}
According to this definition we have a chain of strict inclusions
\begin{equation}\label{inc}
\overline{\Delta}_{i_M}\supset\overline{\Delta}_{i_{M-1}}\supset\ldots\supset\overline{\Delta}_{i_0},
\end{equation}
such that $\overline{\Delta}_{i_M}=\Delta$, $\overline{\Delta}_{0}=\Delta_0=\{\alpha \in \Delta: s\alpha=\alpha\}$ is the set of roots fixed by the action of $s$, and $\overline{\Delta}_{i_k}\setminus \overline{\Delta}_{i_{k-1}}={\Delta}_{i_k}$.

Let ${\Delta}_{i_k}^r$ be the subset of roots in $({\Delta}_{i_k})_+$ orthogonal projections of which onto $\h_{i_k}$ are directed along a ray $r\subset \h_{i_k}$ starting at the origin. We call ${\Delta}_{i_k}^r$ the family corresponding to the ray $r$.

\begin{lemma}\label{lem}
Each ${\Delta}_{i_k}^r$ is an additively closed set of roots.

Let ${\Delta}_{i_k}^{r_1}$ and  ${\Delta}_{i_k}^{r_2}$ be two families corresponding to rays $r_1$ and $r_2$, and $\delta_1\in {\Delta}_{i_k}^{r_1}$, $\delta_2 \in {\Delta}_{i_k}^{r_2}$ two roots such that $\delta_1+\delta_2 =\delta_3\in \Delta$. Then $\delta_3 \in {\Delta}_{i_k}^{r_3}$, where ${\Delta}_{i_k}^{r_3}$ is the family corresponding to a ray $r_3$ such that $r_3$ lies inside of the angle formed by $r_1$ and $r_2$.
\end{lemma}

\begin{proof}
All statements are simple consequences of the fact that the sum of the orthogonal projections of any two roots onto $\h_{i_k}$ is equal to the orthogonal projection of the sum.

In the first case the orthogonal projections of any two roots $\alpha, \beta$ from ${\Delta}_{i_k}^r$ onto $\h_{i_k}$ have the same direction therefore the orthogonal projection of the sum $\alpha+\beta$ onto $\h_{i_k}$ has the same direction as the orthogonal projections of $\alpha$ and $\beta$, and hence $\alpha+\beta \in {\Delta}_{i_k}^r$.

In the second case it suffices to observe that the sum of the orthogonal projections of $\delta_1$ and $\delta_2$ onto $\h_{i_k}$ is equal to the orthogonal projection of the sum, and the sum of the orthogonal projections of $\delta_1$ and $\delta_2$ onto $\h_{i_k}$ lies inside of the angle formed by $r_1$ and $r_2$.

\end{proof}

Now we construct an axillary normal ordering on $\Delta_+^s$ by induction starting from the set $({\Delta}_{i_0})_+$ as follows.

If $i_0=0$ or $\h_{i_0}$ is one--dimensional then we fix an arbitrary normal order on $({\Delta}_{i_0})_+$.

If $\h_{i_0}$ is two--dimensional then we choose a normal ordering in $({\Delta}_{i_0})_+$ in the following way.
First fix an initial arbitrary normal ordering on $({\Delta}_{i_0})_+$. Since by Lemma \ref{lem} each set ${\Delta}_{i_0}^{r}$ is additively closed we obtain an induced ordering for ${\Delta}_{i_0}^{r}$ which satisfies the defining property for the normal ordering.

Now using these induced orderings on the sets ${\Delta}_{i_0}^{r}$ we define an axillary normal ordering on $({\Delta}_{i_0})_+$ such that on the sets ${\Delta}_{i_0}^{r}$ it coincides with the induced normal ordering defined above, and
if ${\Delta}_{i_0}^{r_1}$ and  ${\Delta}_{i_0}^{r_2}$ are two families corresponding to rays $r_1$ and $r_2$ such that $r_2$ lies on the right from $r_1$ in $\h_{i_0}$ then for any $\alpha\in {\Delta}_{i_0}^{r_1}$ and  $\beta \in {\Delta}_{i_0}^{r_2}$ one has $\alpha<\beta$. By Lemma \ref{lem} the two conditions imposed on the axillary normal ordering in $({\Delta}_{i_0})_+$ are compatible and define it in a unique way for the given initial normal ordering on $({\Delta}_{i_0})_+$. Since $s$ acts by a clockwise rotation on $\h_{i_0}$ we have $s({\Delta}_{i_0}^{r})={\Delta}_{i_0}^{s(r)}$ for $s(r)$ in the upper--half plane, and hence the new normal ordering satisfies the condition that for any $\alpha\in ({\Delta}_{i_0})_+$ such that $s\alpha \in ({\Delta}_{i_0})_+$ one has $s\alpha>\alpha$.

Now assume that an axillary normal ordering has already been constructed for the set $\overline{\Delta}_{i_{k-1}}$ and define it for the set $\overline{\Delta}_{i_k}$. Observe that by construction $\overline{\Delta}_{i_{k-1}}\subset \overline{\Delta}_{i_k}$ is the root system of a Levi subalgebra $\g_{i_{k-1}}$ inside of the Levi subalgebra $\g_{i_k}$ of $\g$ with the root system $\overline{\Delta}_{i_k}$. Indeed, the Levi subalgebras $\g_{i_k}$ can be defined inductively starting from $\g_{i_M}=\g$ so that $\g_{i_{k-1}}$ is the centralizer of $h_{i_k}$ in $\g_{i_k}$.  In particular, $\overline{\Delta}_{i_{k-1}}$ is generated by some subset of simple roots of the set of simple roots of $(\overline{\Delta}_{i_{k}})_+$. Therefore there exists an initial normal ordering on $(\overline{\Delta}_{i_{k}})_+$ in which the roots from the set $(\overline{\Delta}_{i_k})_+\setminus (\overline{\Delta}_{i_{k-1}})_+=({\Delta}_{i_k})_+$ form an initial segment and the remaining roots from $(\overline{\Delta}_{i_{k-1}})_+$ are ordered according to the previously defined axillary normal ordering. As in case of the induction base this initial normal ordering gives rise to an induced ordering on each set ${\Delta}_{i_k}^r$.

Now using these induced orderings on the sets ${\Delta}_{i_k}^{r}$ we define an axillary normal ordering on $(\overline{\Delta}_{i_{k}})_+$. We impose the following conditions on it. Firstly we require that the roots from the set $({\Delta}_{i_k})_+$ form an initial segment and the remaining roots from $(\overline{\Delta}_{i_{k-1}})_+$ are ordered according to the previously defined axillary normal ordering.
Secondly, on the sets ${\Delta}_{i_k}^{r}$ the axillary normal ordering coincides with the induced normal ordering defined above, and if ${\Delta}_{i_k}^{r_1}$ and  ${\Delta}_{i_k}^{r_2}$ are two families corresponding to rays $r_1$ and $r_2$ such that $r_2$ lies on the right from $r_1$ in $\h_{i_k}$ then for any $\alpha\in {\Delta}_{i_k}^{r_1}$ and  $\beta \in {\Delta}_{i_k}^{r_2}$ one has $\alpha<\beta$. By Lemma \ref{lem} the conditions imposed on the axillary normal ordering in $({\Delta}_{i_k})_+$ are compatible and define it in a unique way. Since $s$ acts by a clockwise rotation on $\h_{i_k}$ we have $s({\Delta}_{i_k}^{r})={\Delta}_{i_k}^{s(r)}$ for $s(r)$ in the upper--half plane. Note also that the roots from $\overline{\Delta}_{i_{k-1}}$ have zero orthogonal projections onto $\h_{i_k}$. Therefore the new normal ordering satisfies the following property.

\begin{lemma}\label{lem0}
For any $\alpha\in ({\Delta}_{i_k})_+$ such that $s\alpha \in ({\Delta}_{i_k})_+$ one has $s\alpha>\alpha$ and if $\beta, \gamma\in \overline{\Delta}_{i_{k-1}}$, $s\alpha+\beta, \alpha+\gamma\in \Delta$ then $s\alpha+\beta,\alpha+\gamma\in \Delta_+^s$ and $s\alpha+\beta>\alpha+\gamma$.
\end{lemma}

Now we proceed by induction and obtain an axillary normal ordering on $\Delta_+^s$.

Observe that, according to the definition of the axillary normal ordering of $\Delta_+^s$ constructed above we have the following properties of this normal ordering.
\begin{lemma}\label{lem1}
For any $\alpha\in {\Delta}_{i_{k}}$ and $\beta\in {\Delta}_{i_{k+1}}$ we have $\alpha>\beta$, and if ${\Delta}_{i_k}^{r_1}$ and  ${\Delta}_{i_k}^{r_2}$ are two families corresponding to rays $r_1$ and $r_2$ such that $r_2$ lies on the right from $r_1$ in $\h_{i_k}$ then for any $\alpha\in {\Delta}_{i_k}^{r_1}$ and  $\beta \in {\Delta}_{i_k}^{r_2}$ one has $\alpha<\beta$. Moreover, the roots from the sets ${\Delta}_{i_k}^{r}$ form minimal segments, and the roots from the set $(\Delta_0)_+$ form a final minimal segment.
\end{lemma}

The involutions $s^1$ and $s^2$ act in $\h_{i_k}$ as reflections with respect to the lines orthogonal to the vectors labeled by $v^1_k$ and $v^2_k$, respectively, at Figure 1, the angle between $v^1_k$ and $v^2_k$ being equal to $\pi-(\varphi_{i_k}+\psi_{i_k})$. The nonzero projections of the roots from the set $\{\gamma_1, \ldots \gamma_n\}\bigcap {\Delta}_{i_k}$ onto the plane $\h_{i_k}$ have the same (or the opposite) direction as the vector $v^1_k$, and the nonzero projections of the roots from the set $\{\gamma_{n+1}, \ldots , \gamma_{l'}\}\bigcap {\Delta}_{i_k}$ onto the plane $\h_{i_k}$ have the same (or the opposite) direction as the vector $v^2_k$.

For each of the involutions $s^1$ and $s^2$ we obviously have decompositions $\Delta_{s^{1,2}}=\bigcup_{k=0}^M{{\Delta}_{i_k}^{1,2}}$, $\Delta_{s}=\bigcup_{k=0}^M{{\Delta}_{i_k}^{s}}$, where ${{\Delta}_{i_k}^{1,2}}={{\Delta}_{i_k}}\bigcap \Delta_{s^{1,2}}$, $\Delta_{s^{1,2}}=\{\alpha \in \Delta_+^s:s^{1,2}\alpha \in -\Delta_+^s\}$, ${{\Delta}_{i_k}^{s}}={{\Delta}_{i_k}}\bigcap \Delta_{s}$, ${{\Delta}_{i_k}^{s}}={{\Delta}_{i_k}^{2}}\bigcup s^2{{\Delta}_{i_k}^{1}}$. In the plane $\h_{i_k}$, the elements from the sets ${{\Delta}_{i_k}^{1,2}}$ are projected onto the interiors of the sectors labeled by ${{\Delta}_{i_k}^{1,2}}$ and the elements from the set ${{\Delta}_{i_k}^{s}}$ are projected onto the interior of the union of the sectors labeled by ${{\Delta}_{i_k}^{2}}$ and $s^2{{\Delta}_{i_k}^{1}}$. Therefore the sets ${{\Delta}_{i_k}^{1}}$ and ${{\Delta}_{i_k}^{2}}$ have empty intersection and are the unions of the sets ${\Delta}_{i_k}^{r}$ with $r$ belonging to the sector ${{\Delta}_{i_k}^{1,2}}$, and the sets ${{\Delta}_{i_k}^{s}}$ have empty intersection and are the unions of the sets ${\Delta}_{i_k}^{r}$ with $r$ belonging to the union of the sectors labeled by ${{\Delta}_{i_k}^{2}}$ and $s^2{{\Delta}_{i_k}^{1}}$.

For any root $\alpha\in \Delta_{s^{1}}$ one obviously has $\alpha \in {{\Delta}_{i_k}^{1}}$, where $\h_{i_k}$ is a two--dimensioinal plane, as by the assumption imposed  before (\ref{cond2}) there are no one--dimensional subspaces $\h_{i_k}$ on which $s^1$ acts by multiplication by $-1$.
Thus in case if $\h_{i_k}$ is an invariant line on which $s$ acts by multiplication by $-1$ the set ${{\Delta}_{i_k}^{1}}$ is empty and hence ${{\Delta}_{i_k}^{s}}={{\Delta}_{i_k}^{2}}$. This set is the set ${\Delta}_{i_k}^{r}=({\Delta}_{i_k})_+$, where $r$ is the positive semi-axis in $\h_{i_k}$. From the last two observations we deduce that the sets $\Delta_{s^{1}}$ and $\Delta_{s^{2}}$ have always empty intersection. In particular, by the results of \S 3 in \cite{Z1} the decomposition $s=s^1s^2$ is reduced in the sense that $l(s)=l(s^2)+l(s^1)$, and $\Delta_{s}=\Delta_{s^{2}}\bigcup s^2(\Delta_{s^{1}})$ (disjoint union), where $\Delta_{s}=\{\alpha \in \Delta_+^s:s\alpha \in -\Delta_+^s\}$.

\begin{lemma}\label{lem2}
Assume that $\Delta_+^s$ is equipped with an arbitrary normal ordering such that the roots from the set ${\Delta}_{i_k}^r=\{\delta_1,\ldots, \delta_a\}$ form a minimal segment $\delta_1,\ldots, \delta_a$, for some $1\leq p<k$ the roots from the set ${\Delta}_{i_p}^t=\{\xi_1,\ldots, \xi_b\}$ form a minimal segment $\xi_1,\ldots, \xi_b$ and the segment $\delta_1,\ldots, \delta_a,\xi_1,\ldots, \xi_b$ is also minimal. Then applying elementary transposition one can reduce the last segment to the form $\xi_{i_1},\ldots, \xi_{i_b},\delta_{j_1},\ldots, \delta_{j_a}$.
\end{lemma}

\begin{proof}
The proof is by induction. First consider the minimal segment $\delta_1,\ldots, \delta_a,\xi_1$.

Since the orthogonal projection of the roots from the set ${\Delta}_{i_p}$ onto $\h_{i_k}$ are equal to zero, for any $\alpha\in {\Delta}_{i_p}^t$ and $\beta \in {\Delta}_{i_k}^r$ such that $\alpha+\beta\in \Delta$ we have $\alpha+\beta\in {\Delta}_{i_k}^r$. Assume now that $\alpha$ and $\beta$ are contained in an ordered segment of form (\ref{rank2}) or in a segment with the inverse ordering. By the above observation this segment contains no other roots from ${\Delta}_{i_p}^t$, and $\alpha$ is the first or the last element in that segment. For the same reason the other roots in that segment must also belong to ${\Delta}_{i_k}^r$. Therefore applying an elementary transposition, if necessarily, one can move $\alpha$ to the first position in that segment.

Applying this procedure iteratively to the segment $\delta_1,\ldots, \delta_a,\xi_1$ we can reduce it to the form $\xi_1,\delta_{k_1},\ldots, \delta_{k_a}$.

Now we can apply the same procedure to the segment $\delta_{k_1},\ldots, \delta_{k_a}\xi_2$ to reduce the segment $\xi_1,\delta_{k_1},\ldots, \delta_{k_a}\xi_2$ to the form $\xi_1\xi_2\delta_{l_1},\ldots, \delta_{l_a}$.

Iterating this procedure we obtain the statement of the lemma.

\end{proof}

Now observe that according to Lemma \ref{lem1} the roots from each of the sets $({\Delta}_{i_{k}})_+$ form a minimal segment in the axillary normal ordering of $\Delta_+^s$, and the roots from the sets ${\Delta}_{i_k}^{r}$ form minimal segments inside $({\Delta}_{i_{k}})_+$. As we observed above the sets ${{\Delta}_{i_k}^{1,2}}$ are the unions of the sets ${\Delta}_{i_k}^{r}$ with $r$ belonging to the sectors ${{\Delta}_{i_k}^{1,2}}$ and hence by Lemma \ref{lem1} the roots from the sets ${{\Delta}_{i_k}^{1,2}}$ form an initial and a final segment  inside $({\Delta}_{i_{k}})_+$.

Therefore we can apply Lemmas \ref{lem1} and \ref{lem2} to move all roots from the segments ${{\Delta}_{i_k}^{1}}$ to the left and to move all roots from the segments ${{\Delta}_{i_k}^{2}}$ to the right to positions preceding the final segment formed by the roots from $(\Delta_0)_+$.

Now using similar arguments the roots from the sets $s^2{{\Delta}_{i_k}^{1}}$ forming minimal segments  by Lemma \ref{lem1} as well can be moved to the right to positions preceding the final segment formed by the roots from the set $\Delta_{s^{2}}\bigcup(\Delta_0)_+$.

Note that according to the algorithm given in Lemma \ref{lem2} for each fixed $k$ the mutual positions of the minimal segments formed by the roots from the sets ${\Delta}_{i_k}^{r}$ are preserved by the transpositions used in that lemma. Therefore the new normal ordering obtained this way still satisfies the second property mentioned in Lemma \ref{lem1}, i.e. if ${\Delta}_{i_k}^{r_1}$ and  ${\Delta}_{i_k}^{r_2}$ are two families corresponding to rays $r_1$ and $r_2$ such that $r_2$ lies on the right from $r_1$ in $\h_{i_k}$ then for any $\alpha\in {\Delta}_{i_k}^{r_1}$ and  $\beta \in {\Delta}_{i_k}^{r_2}$ one has $\alpha<\beta$.

Now we can apply elementary transpositions used in the proof Proposition 5.1 in \cite{S10} to bring the initial segment formed by the roots from $\Delta_{s^{1}}$ and the segment formed by the roots from $\Delta_{s^{2}}$ and preceding the final segment $(\Delta_0)_+$ to the form described in (\ref{NO}). These elementary transpositions do not affect positions of other roots. We claim that for $\alpha\in {\Delta}_{i_k}^{r}$ we still have $s\alpha>\alpha$ if $s\alpha\in \Delta_+^s$.

Indeed, if $\alpha \in ({\Delta}_{i_k})_+$, $\alpha \not \in {{\Delta}_{i_k}^{1}}$, $s\alpha \not\in {{\Delta}_{i_k}^{2}}$, $s\alpha\in ({{\Delta}_{i_k}})_+$ this follows from the second property mentioned in Lemma \ref{lem1}.

If $\alpha \in {{\Delta}_{i_k}^{1}}$ and $s\alpha\in \Delta_+^s$ then $s\alpha \not \in \Delta_{s^1}$ as $s^1(s\Delta_{s^1})=s^2(\Delta_{s^1})\subset \Delta_+^s$ since the decomposition $s=s^1s^2$ is reduced. Therefore $s\alpha>\alpha$ as the roots from the set $\Delta_{s^1}$ form an initial segment in the normal ordering of $\Delta_+^s$.

If $\alpha \in ({\Delta}_{i_k})_+$, $\alpha \not \in {{\Delta}_{i_k}^{1}}$, $s\alpha \in {{\Delta}_{i_k}^{2}}$ then $\alpha\not \in \Delta_s$ and $s\alpha\in \Delta_s$ as ${{\Delta}_{i_k}^{2}}\subset \Delta_s$. Therefore $s\alpha>\alpha$ as the roots from the set $\Delta_s \bigcup (\Delta_0)_+$ form a final segment in the normal ordering of $\Delta_+^s$ and $\alpha$ does not belong to that segment.

Finally if $\alpha\in {{\Delta}_{i_k}^{2}}$ then $s\alpha \in -\Delta_+^s$.

Moreover, similar arguments together with the fact that all roots from $\overline{\Delta}_{i_{k-1}}$ have zero orthogonal projections onto $\h_{i_k}$ show that the new normal ordering still satisfies the property of Lemma \ref{lem0}. Note that for one--dimensional $\h_{i_k}$ this property is void.

By construction the roots from the set $\Delta_s$ form a minimal segment in $\Delta_+^s$ of the form $\gamma, \ldots , \beta_{l(s^2)}^2$ which contains $\Delta_{s^2}$. This completes the proof of the proposition.

\end{proof}

We call normal ordering (\ref{NO}) a normal ordering associated to $s$.

Later we shall use the circular normal ordering of $\Delta$ corresponding to the system of positive roots $\Delta_+^s$ and to its normal ordering introduced in Proposition \ref{pord}.

Let
\begin{equation*}
{\Delta}_0=\{\alpha\in \Delta|s(\alpha)=\alpha\},
\end{equation*}
and $\Gamma$ the set of simple roots in $\Delta_+^s$.
We shall need the parabolic subalgebra $\p$ of $\g$ and the parabolic subgroup $P$ associated to the subset $\Gamma_0=\Gamma\bigcap {\Delta}_{0}$ of simple roots. Let $\n$ and $\l$ be the nilradical and the Levi factor of $\p$, $N$ and $L$ the unipotent radical and the Levi factor of $P$, respectively.

Note that we have natural inclusions of Lie algebras $\p\supset\n$, and ${\Delta}_{0}$ is the root system of the reductive Lie algebra $\l$. We also denote by $\opn$ the nilpotent subalgebra opposite to $\n$ and by $\overline{N}$ the subgroup in $G$ corresponding to $\opn$.
The linear subspace of $\g$ generated by the root vectors $X_{\alpha}$ ($X_{-\alpha}$), $\alpha\in \Delta_{\m_+}$ is in fact a Lie subalgebra ${\m_+}\subset \g$ (${\m_-}\subset \g$).
Note that by definition $\Delta_{\m_+} \subset \Delta_+^s$, and hence ${\m_\pm}\subset \b_\pm^s$, where $\b_+^s$ is the Borel subalgebra associated to $\Delta_+^s$ and $\b_-^s$ is the opposite Borel subalgebra. Let $\n_\pm^s$ be the nilradicals of $\b_\pm^s$. Denote by $B_\pm^s$ the corresponding Borel subgroups and by $N_\pm^s$ their unipotent radicals. Let also $M_\pm\subset G$ be the subgroups corresponding to the Lie subalgebras $\m_\pm$.

Introduce the element $\bar{h}_0=\sum_{k=1}^{M}h_{i_k}\in \h_\mathbb{R}$. Let $\h_0^*\subset \h_{\mathbb{R}}^*$ be the image of $\h_0$ in $\h_{\mathbb{R}}^*$ under the isomorphism $\h_{\mathbb{R}}^*\simeq \h_{\mathbb{R}}$ induced by the Killing form. By the definition of $\Delta_+^s$ for any $x\in \h_0^*$ one has $\bar{h}_0(x)=0$ and a root $\alpha\in \Delta\setminus \Delta_0$ is positive iff $\bar{h}_0(\alpha)>0$.

Denote a representative for the Weyl group element $s$ in $G$ by the same letter.
Let $Z$ be the subgroup of $G$ generated by the semisimple part of the Levi subgroup $L$ and by the centralizer of $s$ in $H$. The level surface $\mathcal{C}$ is the variety $NZ{s^{-1}}N$.

The following proposition is a modification of Propositions 2.1 and 2.2 in \cite{S6}.
\begin{proposition}\label{prop2}
Let $N_s=\{ v \in N|svs^{-1}\in \overline{N} \}$. Then
the conjugation map
\begin{equation}\label{cross}
N\times N_sZ{s^{-1}}\rightarrow NZ{s^{-1}}N
\end{equation}
is an isomorphism of varieties. Moreover, the variety $N_sZ{s^{-1}}$ is a transversal slice to the set of conjugacy classes in $G$.
\end{proposition}

The variety $\Sigma=N_sZ{s^{-1}}\simeq NZ{s^{-1}}N/N$ is a subvariety in $G$. It is an analogue of the Slodowy slices in algebraic group theory.

The operator $\Pi$ will be defined in terms of certain functions on $NZ{s^{-1}}N$.
These functions are related to some natural coordinates on Bruhat cells in $G$. First, as a warm up exercise, we obtain explicit formulas for these coordinates.

If $w\in W$ is any Weyl group element we shall denote a representative of $w$ in $G$ by the same letter. Let $\Delta_+$ be any system of positive roots, $\Delta_-=-\Delta_+$, $B_\pm$ and $N_\pm$ the corresponding Borel subgroups in $G$ and their unipotent radicals, $\b_\pm$ and $\n_\pm$ their Lie algebras, respectively. Denote $N_{w^{-1}}=\{n\in N_+:w^{-1}nw\in N_-\}$. $N_{w^{-1}}$ is a subgroup in $N_+$ generated by one parametric subgroups corresponding to the roots from the set $\Delta_{w^{-1}}=\{\alpha \in \Delta_+:w^{-1}\alpha\in \Delta_-\}$. Denote by $s_i$ the reflection with respect to a simple root $\alpha_i\in \Delta_+$, $i=1,\ldots ,l$ and let $w=s_{i_1}\ldots s_{i_k}$ be a reduced decomposition of $w$. Then $\Delta_{w^{-1}}=\{\beta_1, \ldots \beta_k\}$, and $\beta_j=s_{i_1}\ldots s_{i_{j-1}}\alpha_j$. Note that the elements $w_j=s_{i_1}\ldots s_{i_{j}}$ can also be represented in the form $w_j=s_{\beta_j}\ldots s_{\beta_1}$, where for any root $\alpha \in \Delta$ we denote by $s_\alpha$ the reflection with respect to $\alpha$. Observe also that $\Delta_{w_j^{-1}}=\{\beta_1, \ldots, \beta_j\}=\{\alpha \in \Delta_+:w_j^{-1}\alpha\in \Delta_-\}$.

Let $X_\alpha(t)=\exp(tX_\alpha)$, $t\in \mathbb{C}$ and denote by $N_\alpha$ the one parametric subgroup corresponding to root $\alpha$, so $X_\alpha(t)\in N_\alpha$. Any element of the Bruhat cell $B_+w^{-1}B_+$ can be uniquely represented in the form $nw^{-1}hn_{w^{-1}}$, $n\in N_+, h\in H$, $n_{w^{-1}}=X_{\beta_k}(q_k)\ldots X_{\beta_1}(q_1)$.

Fix a normal ordering $\beta_1,\ldots ,\beta_D$ of the system of positive roots $\Delta_+$ such that $\beta_1,\ldots , \beta_k$ is its initial segment and equip $\Delta$ with the corresponding circular normal ordering. This is always possible by the results of \S 3 in \cite{Z1}.

Let $\omega_i$, $i=1,\ldots, l$ be the fundamental weights of $\g$ corresponding to $\Delta_+$, $V_{\omega_i}$ the irreducible representation of $G$ with highest weight $\omega_i$, $v_{\omega_i}\in V_{\omega_i}$ a non--zero highest weight vector. Denote by $(\cdot, \cdot )$ the contravariant bilinear non--degenerate form on $V_{\omega_i}$ such that $(v_{\omega_i},v_{\omega_i})=1$ and  $(\omega(g)v,w)=(v,gw)$, where $\omega$ is the Chevalley anti--involution on $G$. For any reflection $s_\alpha$ one can fix a representative $s_\alpha$ in $G$ such that $\omega(s_\alpha)=s_\alpha^{-1}$. It suffices to do that for simple reflections and one can put $s_i=\exp(f_i)\exp(-e_i)\exp(f_i)$, where $e_i\in \n_+,f_i\in \n_-,h_i\in \h$ are the Chevalley generators of $\g$ on which $\omega$ acts as follows $\omega(f_i)=e_i, \omega(e_i)=f_i, \omega(h_i)=h_i$.

If $\alpha, \beta\in \Delta$ are such that the segment $[\alpha,\beta]$ is minimal then we denote by $N_{[\alpha,\beta]}$ the subgroup in $G$ generated by the one--parametric subgroups corresponding to the roots from $[\alpha,\beta]$.

\begin{proposition}\label{Br}
Let $g=nw^{-1}hn_{w^{-1}}\in B_+w^{-1}B_+$, $n\in N_+, h\in H$, $n_{w^{-1}}=X_{\beta_k}(q_k)\ldots X_{\beta_p}(q_p)$, $1\leq p\leq k$ be an element of the Bruhat cell $B_+w^{-1}B_+$. Then
\begin{equation}\label{tp}
q_p=c_p\frac{(w_{p-1}v_{\omega_{i_p}},wgw_p v_{\omega_{i_p}})}{(w_{p-1}v_{\omega_{i_p}},wgw_{p-1} v_{\omega_{i_p}})},
\end{equation}
where $w_p=s_{\beta_p}\ldots s_{\beta_1}$, $w_{p-1}=s_{\beta_{p-1}}\ldots s_{\beta_1}$, $c_p$ is a non--zero constant only depending on the choice of the representative $s_{\beta_p}\in G$ and on the choice of the root vector $X_{\beta_p}$, and it is assumed that $w_0=1$.
\end{proposition}

\begin{proof}
First, by Lemma \ref{minsegm} and using commutation relations between one--parametric subgroups we can write $n_{w^{-1}}=X_{\beta_k}(q_k)\ldots X_{\beta_p}(q_p)=X_{\beta_p}(q_p)m_1$, $m_1\in N_{[\beta_{p+1}, \beta_k]}$. Since $\Delta_{w_p^{-1}}=\{\beta_1, \ldots, \beta_p\}$ we have $w_p^{-1}N_{[\beta_{p+1}, \beta_k]}w_p\subset N_+$, and hence
$$n_{w^{-1}}w_p v_{\omega_{i_p}}=X_{\beta_p}(q_p)m_1w_p v_{\omega_{i_p}}=X_{\beta_p}(q_p)w_pw_p^{-1}m_1w_p v_{\omega_{i_p}}=X_{\beta_p}(q_p)w_pv_{\omega_{i_p}}$$ as $v_{\omega_{i_p}}$ is a highest weight vector. We deduce that
$$
(w_{p-1}v_{\omega_{i_p}},wgw_p v_{\omega_{i_p}})=(w_{p-1}v_{\omega_{i_p}},wnw^{-1}hX_{\beta_p}(q_p)w_pv_{\omega_{i_p}}).
$$

Now observe that $wnw^{-1}\in N_{[\beta_{k+1}, -\beta_k]}$ and that $H$ normalizes each one--parametric subgroup $N_\alpha$. Therefore using Lemma \ref{minsegm} and commutation relations between one--parametric subgroups one can uniquely factorize the element $wnw^{-1}h$ as $wnw^{-1}h=m_2hm_3$, $m_2\in N_{[-\beta_p,-\beta_k]}$, $m_3\in N_{[\beta_{k+1},-\beta_{p-1}]}$. Since $\omega(m_2)\in N_{[\beta_p,\beta_k]}$, $w_{p-1}^{-1}N_{[\beta_{p}, \beta_k]}w_{p-1}\subset N_+$ and $v_{\omega_{i_p}}$ is a highest weight vector we obtain
$$
(w_{p-1}v_{\omega_{i_p}},wgw_p v_{\omega_{i_p}})=(\omega(m_2)w_{p-1}v_{\omega_{i_p}},hm_3X_{\beta_p}(q_p)w_pv_{\omega_{i_p}})=
(w_{p-1}v_{\omega_{i_p}},hm_3X_{\beta_p}(q_p)w_pv_{\omega_{i_p}}).
$$

Observe also that using Lemma \ref{minsegm} and commutation relations between one--parametric subgroups we can a uniquely factorize the element $m_3X_{\beta_p}(q_p)\in N_{[\beta_{p},-\beta_{p-1}]}$ as $m_3X_{\beta_p}(q_p)=X_{\beta_p}(q_p)m_4m_5$,
$m_4\in N_{[\beta_{p+1},\beta_{D}]}$, $m_5\in N_{[-\beta_{1},-\beta_{p-1}]}$. Finally remark that $w_{p}^{-1}N_{[-\beta_{1},-\beta_{p-1}]}w_p\subset N_+$ and $w_{p}^{-1}N_{[\beta_{p+1},\beta_{D}]}w_p\subset N_+$ as $\Delta_{w_p^{-1}}=\{\beta_1, \ldots, \beta_p\}$ and hence
$$
(w_{p-1}v_{\omega_{i_p}},wgw_p v_{\omega_{i_p}})=(w_{p-1}v_{\omega_{i_p}},hX_{\beta_p}(q_p)m_4m_5w_pv_{\omega_{i_p}})=
(w_{p-1}v_{\omega_{i_p}},hX_{\beta_p}(q_p)w_pv_{\omega_{i_p}})
$$
as $v_{\omega_{i_p}}$ is a highest weight vector.

Note that the vector $w_{p-1}v_{\omega_{i_p}}$ has weight $w_{p-1}{\omega_{i_p}}$. Since weight spaces corresponding to different weights are orthogonal with respect to the contravariant bilinear non--degenerate form on $V_{\omega_{i_p}}$ only terms of weight $w_{p-1}{\omega_{i_p}}$ in the element $hX_{\beta_p}(q_p)w_pv_{\omega_{i_p}}$ will give non--trivial contributions to the scalar product $(w_{p-1}v_{\omega_{i_p}},hX_{\beta_p}(q_p)w_pv_{\omega_{i_p}})$. To find these terms we observe that each weight space is an eigenspace for the action of $H$ and that $X_{\beta_p}(q_p)w_pv_{\omega_{i_p}}=X_{\beta_p}(q_p)s_{\beta_p}w_{p-1}v_{\omega_{i_p}}$. Since $\beta_p\not\in \Delta_{w_{p-1}^{-1}}$ we infer $X_{\beta_p}w_{p-1}v_{\omega_{i_p}}=0$, and hence $w_{p-1}v_{\omega_{i_p}}$ is a highest weight vector for the ${\mathfrak s l}_2$--triple generated by the elements $X_{\pm\beta_p}$. Moreover, since $w_{p-1}^{-1}(-\beta_p)=-\alpha_{i_p}$ the vector $X_{-\alpha_{i_p}}=w_{p-1}^{-1}X_{-\beta_p}w_{p-1}$ is a root vector corresponding to $-\alpha_{i_p}$, and hence $X_{-\beta_p}^2w_{p-1}v_{\omega_{i_p}}=w_{p-1}X_{-\alpha_{i_p}}^2v_{\omega_{i_p}}=0$ by the definition of $v_{\omega_{i_p}}$. Therefore $w_{p-1}v_{\omega_{i_p}}$ is a highest weight vector for the two--dimensional irreducible representation of the ${\mathfrak s l}_2$--triple generated by the elements $X_{\pm\beta_p}$, and $s_{\beta_p}w_{p-1}v_{\omega_{i_p}}$ is a non--zero lowest weight vector for that representation. Recalling the standard ${\mathfrak s l}_2$--representation theory we deduce that
$$
X_{\beta_p}(q_p)w_pv_{\omega_{i_p}}=X_{\beta_p}(q_p)s_{\beta_p}w_{p-1}v_{\omega_{i_p}}=
s_{\beta_p}w_{p-1}v_{\omega_{i_p}}+q_pX_{\beta_p}s_{\beta_p}w_{p-1}v_{\omega_{i_p}}=
s_{\beta_p}w_{p-1}v_{\omega_{i_p}}+\frac{q_p}{c_p}w_{p-1}v_{\omega_{i_p}},
$$
where $c_p$ is a non--zero constant only depending on the choice of the representative $s_{\beta_p}\in G$ and on the choice of the root vector $X_{\pm\beta_p}$.

The only term of weight $w_{p-1}{\omega_{i_p}}$ in the right hand side of the last identity is $\frac{q_p}{c_p}w_{p-1}v_{\omega_{i_p}}$, and hence
$$
(w_{p-1}v_{\omega_{i_p}},wgw_p v_{\omega_{i_p}})=
(w_{p-1}v_{\omega_{i_p}},hX_{\beta_p}(q_p)w_pv_{\omega_{i_p}})=(w_{p-1}v_{\omega_{i_p}},h\frac{q_p}{c_p}w_{p-1}v_{\omega_{i_p}}),
$$
and
$$
q_p=c_p\frac{(w_{p-1}v_{\omega_{i_p}},wgw_p v_{\omega_{i_p}})}{(w_{p-1}v_{\omega_{i_p}},hw_{p-1}v_{\omega_{i_p}})}.
$$

Similar arguments show that
$$
(w_{p-1}v_{\omega_{i_p}},hw_{p-1}v_{\omega_{i_p}})=(w_{p-1}v_{\omega_{i_p}},wgw_{p-1}v_{\omega_{i_p}}).
$$

Combining the last two identities we obtain formula (\ref{tp}).

\end{proof}

Let $g=nw^{-1}hn_{w^{-1}}\in B_+w^{-1}B_+$, $n\in N_+, h\in H$, $n_{w^{-1}}=X_{\beta_k}(q_k)\ldots X_{\beta_1}(q_1)$ be an arbitrary element. Using formula (\ref{tp}) one can find the numbers $q_p$ inductively. Namely,
$$
q_1=c_1\frac{(v_{\omega_{i_1}},wgw_1 v_{\omega_{i_1}})}{(v_{\omega_{i_1}},wg v_{\omega_{i_1}})},
$$
and if $q_1,\ldots, q_{p-1}$ are already found then
\begin{equation}\label{c1}
q_p=c_p\frac{(w_{p-1}v_{\omega_p},wgX_{\beta_1}(-q_1)\ldots X_{\beta_{p-1}}(-q_{p-1})w_p v_{\omega_p})}{(w_{p-1}v_{\omega_p},wgX_{\beta_1}(-q_1)\ldots X_{\beta_{p-1}}(-q_{p-1})w_{p-1} v_{\omega_p})}.
\end{equation}

Note also that once the numbers $q_p$ are found, and hence the element $n_{w^{-1}}$ is determined one can also find $n$ using Proposition 2.11 in \cite{FZ}.

Indeed, let $\overline{w}=s_{i_1}\ldots s_{i_D}$ be the reduced decomposition of the longest element of the Weyl group corresponding to a normal ordering $\beta_1,\ldots ,\beta_D$ of the system of positive roots $\Delta_+$. Then $\beta_j=s_{i_1}\ldots s_{i_{j-1}}\alpha_j$. Consider the elements $w_j=s_{i_1}\ldots s_{i_{j}}=s_{\beta_j}\ldots s_{\beta_1}$. Observe that $\Delta_{w_j^{-1}}=\{\beta_1, \ldots, \beta_j\}=\{\alpha \in \Delta_+:w_j^{-1}\alpha\in \Delta_-\}$. The element $n$ can be uniquely represented in the form $n=X_{\beta_1}(r_1)\ldots X_{\beta_D}(r_D)$.
Assume that the root vectors $X_\alpha$ used in the definition of one--parametric subgroups $X_\alpha(t)=\exp(tX_\alpha)$ are chosen in such a way that $\omega(X_\alpha)=X_{-\alpha}$. Then according to Propositions 2.6 and 2.11 in \cite{FZ} we have
\begin{equation}\label{c2}
r_p=d_p\frac{(w_{p-1}v_{\omega_p},nw^{-1}hww_p v_{\omega_p})}{(w_{p}v_{\omega_p},nw^{-1}hww_{p} v_{\omega_p})}=d_p\frac{(w_{p-1}v_{\omega_p},gn_{w^{-1}}^{-1}ww_p v_{\omega_p})}{(w_{p}v_{\omega_p},gn_{w^{-1}}^{-1}ww_{p} v_{\omega_p})},
\end{equation}
where $d_p$ is a non--zero constant only depending on the choice of the representative $s_{\beta_p}\in G$ and on the choice of the root vector $X_{\beta_p}$, and it is assumed that $w_0=1$.

Finally once we know the coefficients $r_i$ we can find $n$ and $h=wn^{-1}gn_{w^{-1}}^{-1}$. Thus we have completely described $g$ in terms of some matrix elements of finite--dimensional irreducible representations. The functions $q_i, r_i$ defined by (\ref{c1}) and (\ref{c2}) together with the fundamental weights evaluated at $h=wn^{-1}gn_{w^{-1}}^{-1}$ can be regarded as natural coordinates on the Bruhat cell $B_+w^{-1}B_+$.

Now we introduce functions required for the definition of the operator $\Pi$.
Let $g\in NZ{s^{-1}}N$. By Proposition \ref{prop2} $g$ can be uniquely represented in the form $g=n^{-1}n_sz{s^{-1}}n$, $n\in N$, $n_s\in N_s$.

\begin{proposition}
Let $g=n^{-1}n_sz{s^{-1}}n\in NZ{s^{-1}}N$. Let $\alpha_i$ be the simple roots of a system of positive roots $\Delta_+^s$ associated to $s$, $s_i$ the corresponding simple reflections, $\beta_1,\ldots ,\beta_D$, $\beta_j=s_{i_1}\ldots s_{i_{j-1}}\alpha_{i_j}$ a normal ordering (\ref{NO}) of $\Delta_+^s$. Denote $\Delta_0\bigcap \Delta_+^s=\{\beta_{d+1},\ldots \beta_D\}$. Let $\omega_i$, $i=1,\ldots l$ be the fundamental weights corresponding to $\Delta_+^s$, $v_{\omega_i}$ a non--zero highest weight vector in the irreducible highest weight representation $V_{\omega_i}$ of highest weight $\omega_i$, $(\cdot,\cdot)$ the contravariant non--degenerate bilinear form on $V_{\omega_i}$ such that $(v_{\omega_i},v_{\omega_i})=1$. Then $n$ can be uniquely factorized as $n=X_{\beta_d}(t_d)\ldots X_{\beta_1}(t_1)$ and the numbers $t_i$ can be found inductively by the following formula
\begin{equation}\label{tind}
t_p=c_p\frac{(w_{p-1}v_{\omega_{i_p}},sg_pw_p v_{\omega_{i_p}})}{(w_{p-1}v_{\omega_{i_p}},sg_pw_{p-1} v_{\omega_{i_p}})},
\end{equation}
where $w_p=s_{\beta_p}\ldots s_{\beta_1}$, $w_{p-1}=s_{\beta_{p-1}}\ldots s_{\beta_1}$, $c_p$ is a non--zero constant only depending on the choice of the representative $s_{\beta_p}\in G$ and on the choice of the root vector $X_{\beta_p}\in \g$, $g_p=n_pgn_p^{-1}$, $n_p=X_{\beta_{p-1}}(t_{p-1})\ldots X_{\beta_1}(t_1)$ and it is assumed that $n_1=1$, $w_0=1$.
\end{proposition}

\begin{proof}
The proof of this proposition is similar to that of the previous proposition. The numbers $t_p$ can be found by induction starting with $p=1$. We shall establish the induction step. The case $p=1$ corresponding to the base of the induction can be considered in a similar way.

Assume that $t_1, \ldots , t_{p-1}$ have already been found. Then
$$sg_p=sn_pgn_p^{-1}=sX_{\beta_{p}}(-t_{p})\ldots X_{\beta_d}(-t_d)n_szs^{-1}X_{\beta_{d}}(t_{d})\ldots X_{\beta_p}(t_p), n_p=X_{\beta_{p-1}}(t_{p-1})\ldots X_{\beta_1}(t_1).$$
By Lemma \ref{minsegm} and using commutation relations between one--parametric subgroups we can write $X_{\beta_{d}}(t_{d})\ldots X_{\beta_p}(t_p)=X_{\beta_p}(t_p)m_1$, $m_1\in N_{[\beta_{p+1}, \beta_d]}$. Since $\Delta_{w_p^{-1}}=\{\beta_1, \ldots, \beta_p\}$ we have $w_p^{-1}N_{[\beta_{p+1}, \beta_d]}w_p\subset N_+^s$, and hence
$$X_{\beta_{d}}(t_{d})\ldots X_{\beta_p}(t_p)w_p v_{\omega_{i_p}}=X_{\beta_p}(t_p)m_1w_p v_{\omega_{i_p}}=X_{\beta_p}(t_p)w_pw_p^{-1}m_1w_p v_{\omega_{i_p}}=X_{\beta_p}(t_p)w_pv_{\omega_{i_p}}$$ as $v_{\omega_{i_p}}$ is a highest weight vector. We deduce that
$$
(w_{p-1}v_{\omega_{i_p}},sg_pw_p v_{\omega_{i_p}})=(w_{p-1}v_{\omega_{i_p}},sX_{\beta_{p}}(-t_{p})\ldots X_{\beta_d}(-t_d)n_szs^{-1}X_{\beta_p}(t_p)w_pv_{\omega_{i_p}}).
$$

Now observe that $X_{\beta_{p}}(-t_{p})\ldots X_{\beta_d}(-t_d)n_s \in N_{[\beta_{p}, \beta_d]}N_s$. Note that the normal ordering in $\Delta_+^s$ associated to $s$ has the property that the set $\Delta_s$ is a segment of the form $\beta_k,\ldots , \beta_d$. Therefore the union $[\beta_{p}, \beta_d]\bigcup [\beta_{k}, \beta_d]$ is also a minimal segment and the subgroup $N_{[\beta_{p}, \beta_d]}N_s$ is generated by one--parametric subgroups corresponding to the roots from that segment.
Therefore using Lemma \ref{minsegm} and commutation relations between one--parametric subgroups one can uniquely factorize the element $X_{\beta_{p}}(-t_{p})\ldots X_{\beta_d}(-t_d)n_s$ as $X_{\beta_{p}}(-t_{p})\ldots X_{\beta_d}(-t_d)n_s=m_2m_3$, $m_2\in N_{[\beta_k,\beta_d]}$, $m_3\in N_{[\beta_{p},\beta_{k-1}]}$, where it is assumed that $N_{[\beta_{p},\beta_{k-1}]}=1$ if $p>k-1$. If $\alpha\in [\beta_{p},\beta_{k-1}]$ then $s\alpha\in \Delta_+^s$ and by the properties of the normal ordering in $\Delta_+^s$ associated to $s$ we have $s\alpha >\alpha$, and if $s\alpha +\alpha_0 \in \Delta_+^s$  for $\alpha_0\in \Delta_0$ then $s\alpha+\alpha_0>\alpha$. Observing also that $Z$ is generated by one--parametric subgroups corresponding to roots from $\Delta_0$ and by the centralizer of $s$ in $H$ which normalizes all one--parametric subgroups corresponding to roots, we deduce $sX_{\beta_{p}}(-t_{p})\ldots X_{\beta_d}(-t_d)n_szs^{-1}=sm_2s^{-1}z_1m_4$,
$m_4=z_1^{-1}sm_3s^{-1}z_1\in N_{[\beta_{p+1},\beta_{D}]}$, $z_1=szs^{-1}\in Z$.

Observe now that using Lemma \ref{minsegm} and commutation relations between one--parametric subgroups we can a uniquely factorize the element $m_4X_{\beta_p}(t_p)\in N_{[\beta_{p},\beta_{D}]}$ as $m_4X_{\beta_p}(t_p)=X_{\beta_p}(t_p)m_5$,
$m_5\in N_{[\beta_{p+1},\beta_{D}]}$. Remark that $w_{p}^{-1}N_{[\beta_{p+1},\beta_{D}]}w_p\subset N_+^s$ as $\Delta_{w_p^{-1}}=\{\beta_1, \ldots, \beta_p\}$ and hence
$$
\begin{array}{l}
(w_{p-1}v_{\omega_{i_p}},sg_pw_p v_{\omega_{i_p}})=(w_{p-1}v_{\omega_{i_p}},sm_2s^{-1}z_1m_4X_{\beta_p}(t_p)w_pv_{\omega_{i_p}})= \\
\\
=(w_{p-1}v_{\omega_{i_p}},sm_2s^{-1}z_1X_{\beta_p}(t_p)m_5w_pv_{\omega_{i_p}})=
(\omega(z_1)\omega(sm_2s^{-1})w_{p-1}v_{\omega_{i_p}},X_{\beta_p}(t_p)w_pv_{\omega_{i_p}}).
\end{array}
$$
as $v_{\omega_{i_p}}$ is a highest weight vector.

By the definition of $m_2$ we have $sm_2s^{-1}\in N_-^s$ , and hence $\omega(sm_2s^{-1})\in N_+^s$. Therefore using arguments similar to those used above we can factorize $\omega(sm_2s^{-1})=m_6m_7$, $m_6\in N_{[\beta_1,\beta_{p-1}]}$, $m_7\in N_{[\beta_p,\beta_{D}]}$ and obtain that
$$
\begin{array}{l}
(w_{p-1}v_{\omega_{i_p}},sg_pw_p v_{\omega_{i_p}})=
(\omega(z_1)\omega(sm_2s^{-1})w_{p-1}v_{\omega_{i_p}},X_{\beta_p}(t_p)w_pv_{\omega_{i_p}})= \\
\\
=(z_2m_6m_7w_{p-1}v_{\omega_{i_p}},X_{\beta_p}(t_p)w_pv_{\omega_{i_p}})=
(z_2m_6w_{p-1}v_{\omega_{i_p}},X_{\beta_p}(t_p)w_pv_{\omega_{i_p}}), z_2=\omega(z_1)\in Z.
\end{array}
$$

As we already showed in the previous proof
$$
X_{\beta_p}(t_p)w_pv_{\omega_{i_p}}=
s_{\beta_p}w_{p-1}v_{\omega_{i_p}}+\frac{t_p}{c_p}w_{p-1}v_{\omega_{i_p}}=
cX_{-\beta_p}w_{p-1}v_{\omega_{i_p}}+\frac{t_p}{c_p}w_{p-1}v_{\omega_{i_p}}, c\in \mathbb{C},
$$
so the first term in the last sum has weight $-\beta_p+w_{p-1}\omega_{i_p}$, and the second one $w_{p-1}\omega_{i_p}$.

The vector $z_2m_6w_{p-1}v_{\omega_{i_p}}$ is a linear combination of vectors of weights of the form $w_{p-1}\omega_{i_p}+\sum_{q=1}^{p-1}c_q\beta_q+\omega_0$, where $\omega_0\in \h_0^*$, $c_q\in \{0,1,2,\ldots \}$.
Since weight spaces corresponding to different weights are orthogonal with respect to the contravariant bilinear non--degenerate form on $V_{\omega_{i_p}}$ the only nontrivial contributions to the product $(z_2m_6w_{p-1}v_{\omega_{i_p}},X_{\beta_p}(t_p)w_pv_{\omega_{i_p}})$ come from the products of vectors of weights either $-\beta_p+w_{p-1}\omega_{i_p}$ or $w_{p-1}\omega_{i_p}$.

In the first case we must have $w_{p-1}\omega_{i_p}+\sum_{q=1}^{p-1}c_q\beta_q+\omega_0=-\beta_p+w_{p-1}\omega_{i_p}$, and hence $\sum_{q=1}^{p-1}c_q\beta_q+\omega_0=-\beta_p$. In particular, $\overline{h}_0(\sum_{q=1}^{p-1}c_q\beta_q+\omega_0)=\sum_{q=1}^{p-1}c_q\overline{h}_0(\beta_q)=-\overline{h}_0(\beta_p)$ which is impossible as $-\overline{h}_0(\beta_p)<0$ and $c_q\overline{h}_0(\beta_q)\geq 0$.

In the second case we must have $w_{p-1}\omega_{i_p}+\sum_{q=1}^{p-1}c_q\beta_q+\omega_0=w_{p-1}\omega_{i_p}$ or $\sum_{q=1}^{p-1}c_q\beta_q+\omega_0=0$. In particular, $\overline{h}_0(\sum_{q=1}^{p-1}c_q\beta_q+\omega_0)=\sum_{q=1}^{p-1}c_q\overline{h}_0(\beta_q)=0$ which forces $c_q=0$ for all $q$ as $\overline{h}_0(\beta_q)>0$ and $c_q\in \{0,1,2,\ldots \}$, and hence $\omega_0=0$ as well.

We conclude that the only nontrivial contributions to the product $(z_2m_6w_{p-1}v_{\omega_{i_p}},X_{\beta_p}(t_p)w_pv_{\omega_{i_p}})$ come from the products of vectors of weights $w_{p-1}\omega_{i_p}$. By the above considerations only terms of the form $z_2w_{p-1}v_{\omega_{i_p}}$ may give contributions of weight $w_{p-1}\omega_{i_p}$ in the weight decomposition of the element $z_2m_6w_{p-1}v_{\omega_{i_p}}$, and this yields
$$
\begin{array}{l}
(w_{p-1}v_{\omega_{i_p}},sg_pw_p v_{\omega_{i_p}})=
(z_2m_6w_{p-1}v_{\omega_{i_p}},X_{\beta_p}(t_p)w_pv_{\omega_{i_p}})= \\
\\
=\frac{t_p}{c_p}(z_2w_{p-1}v_{\omega_{i_p}},w_{p-1}v_{\omega_{i_p}})=
\frac{t_p}{c_p}(w_{p-1}v_{\omega_{i_p}},szs^{-1}w_{p-1}v_{\omega_{i_p}}).
\end{array}
$$

Therefore
$$
t_p=c_p\frac{(w_{p-1}v_{\omega_{i_p}},sg_pw_p v_{\omega_{i_p}})}{(w_{p-1}v_{\omega_{i_p}},szs^{-1}w_{p-1}v_{\omega_{i_p}})}.
$$

Similar arguments show that
$$
(w_{p-1}v_{\omega_{i_p}},szs^{-1}w_{p-1}v_{\omega_{i_p}})=(w_{p-1}v_{\omega_{i_p}},sg_pw_{p-1}v_{\omega_{i_p}}).
$$

Combining the last two identities we obtain formula (\ref{tind}).

\end{proof}

\begin{remark}
Note that in the proof of  Proposition 6.2 in \cite{S12} it is shown that $NZs^{-1}N$ is a closed subvariety in $G$. Since (\ref{cross}) is an isomorphism of varieties the right hand side of (\ref{tind}) is a regular function on $NZs^{-1}N$, and hence the denominator in (\ref{tind}) must be canceled.
\end{remark}

Observing that in the notation of the previous proposition for $g=n^{-1}n_szs^{-1}n$ we have $g_{d+1}=n_szs^{-1}=n_{d+1}gn_{d+1}^{-1}$, $n=n_{d+1}=X_{\beta_{d}}(t_{d})\ldots X_{\beta_1}(t_1)$ and that the map (\ref{cross}) is an isomorphism of varieties we infer the following proposition from the previous statement.
\begin{proposition}\label{proj}
Let $A_p$, $p=1,\ldots,d$ be the rational function on $G$ defined by
$$
A_p(g)=c_p\frac{(w_{p-1}v_{\omega_{i_p}},sgw_p v_{\omega_{i_p}})}{(w_{p-1}v_{\omega_{i_p}},sgw_{p-1} v_{\omega_{i_p}})},
$$
and $Pi_p$ the operator on the space of rational functions on $G$ induced by conjugation by the element $\exp(A_pX_{\beta_p})$,
$$
\Pi_p f(g)=f(\exp(A_p(g)X_{\beta_p})g\exp(-A_p(g)X_{\beta_p})).
$$
Then the composition $\Pi=\Pi_1\circ \ldots \circ \Pi_d$ gives rise to a well--defined operator
$$
\Pi=\Pi_1\circ \ldots \circ \Pi_d:\mathbb{C}[NZs^{-1}N]\rightarrow \mathbb{C}[NZs^{-1}N]^N,
$$
which is a projection operator onto the subspace $[NZs^{-1}N]^N$ of $N$--invariant regular functions on $NZs^{-1}N$.
\end{proposition}


\section{The Poisson structure of q-W algebras}
\label{plg}

\setcounter{equation}{0}
\setcounter{theorem}{0}

In this section we obtain two descriptions of the Poisson structure of the q-W algebras. First following \cite{S6} we recall the definition of Poisson q-W algebras. We shall need some related facts on Poisson--Lie groups and on the definition of the Poisson structure of quasi--triangular Poisson--Lie groups and their dual groups which can be found in \cite{fact,dual,RIMS}.

Let $s\in W$ be an element of the Weyl group $W$ of the pair $(\g,\h)$, $\Delta_+^s$ a system of positive roots associated to $s$, and $\h'$ the orthogonal complement in $\h$, with respect to the Killing form, to the subspace of $\h$ fixed by the natural action of $s$ on $\h$. The restriction of the natural action of $s$ on $\h$ to the subspace $\h'$ has no fixed points. Therefore one can define the Cayley transform ${1+s \over 1-s }P_{{\h'}}$ of the restriction of $s$ to ${\h'}$, where $P_{{\h'}}$ is the orthogonal projection operator onto ${{\h'}}$ in $\h$, with respect to the Killing form.

Let $r\in {\rm End}~\g$ be the endomorphism defined by
\begin{equation}\label{r}
r=P_+-P_--{1+s \over 1-s}P_{{\h'}},
\end{equation}
where $P_+,P_-$ and $P_{{\h'}}$ are the projection operators onto ${\frak n}_+^s,{\frak
n}_-^s$ and ${\frak h}'$ in
the direct sum
$$
{\frak g}={\frak n}_+^s +{\frak h}'+{\h'}^\perp + {\frak n}_-^s,
$$
and ${\h'}^\perp$ is the orthogonal complement to $\h'$ in $\h$ with respect to the Killing form.

The endomorphism $r$ satisfies the modified classical Yang-Baxter
equation
\begin{equation}
\left[ rX,rY\right] -r\left( \left[ rX,Y\right] +\left[
X,rY\right] \right) =-\left[ X,Y\right] ,\;X,Y\in {\frak g}
\label{cybe}
\end{equation}
and this ensures that
\begin{equation}
\left[ X,Y\right] _{*}=\frac 12\left( \left[ rX,Y\right] +\left[
X,rY\right] \right) ,X,Y\in {\frak g},  \label{rbr}
\end{equation}
is a Lie bracket.

Identifying the dual space $\g^*$ with $\g$ using the Killing form one can check that this commutator is dual to a cocycle on $\g$ which makes $(\g,\g^*)$ a Lie bialgebra. Note also that $r$ is skew-symmetric with respect to the Killing form.

Let $G$ be a connected semi--simple Poisson--Lie group with the
tangent Lie bialgebra $({\frak g},{\frak g}^*)$,
$G^*$ the dual connected simply--connected Poisson--Lie group.

Define operators $r_\pm \in {\rm End}\ {\frak g}$ by
\[
r_{\pm }=\frac 12\left( r\pm id\right) .
\]

The classical Yang--Baxter equation implies that $r_{\pm }$ ,
regarded as a mapping from ${\frak g}^{*}$ into ${\frak g}$, is a
Lie algebra homomorphism. Moreover, $r_{+}^{*}=-r_{-},$\ and
$r_{+}-r_{-}=id.$

One can describe the dual group $G^*$ in terms of $G$ as follows.
Put ${\frak {d}}={\frak g + {\g}}$ (direct sum of two copies).
The mapping
\begin{eqnarray}\label{imbd}
{\frak {g}}^{*}\rightarrow {\frak {d}}~~~:X\mapsto
(X_{+},~X_{-}),~~~X_{\pm }~=~r_{\pm }X
\end{eqnarray}
is a Lie algebra embedding. Thus we may identify ${\frak g^{*}}$
with a Lie subalgebra in ${\frak {d}}$.

Naturally, embedding (\ref{imbd}) gives rise to a Lie group embedding
$$
G^*\rightarrow G\times G,~~L\mapsto (L_+,L_-).
$$
We shall identify $G^*$ with the image of this embedding in
$G\times G$.

Now we explicitly describe the Poisson structures on the Poisson--Lie
group $G$ and on its dual group $G^*$.

For every group $A$ with Lie algebra $\frak a$ and any function
$\varphi \in C^\infty (A)$ we define left and right gradients $\nabla
\varphi , \nabla^{\prime} \varphi$, which are $C^\infty$-functions on $A$ with values in  ${\frak a}^{*}$, by the formulae
\begin{eqnarray}
\xi ( \nabla \varphi (x))\ =\left( \frac d{dt}\right) _{t=0}\varphi (e^{t\xi }x),  \nonumber \\
\xi ( \nabla^{\prime} \varphi (x))=\left( \frac d{dt}\right)
_{t=0}\varphi (xe^{t\xi }),~~\xi \in {\frak {a}.}  \label{grad}
\end{eqnarray}

Denote by $\langle \cdot , \cdot \rangle$ the Killing form on $\g$ and identify $\g$ with $\g^*$ using this form. Then the left and the right gradients of functions on $G$ can be regarded as $C^\infty$-functions on $G$ with values in $\g$.
They satisfy the following relations
\begin{eqnarray}
\langle X, \nabla \varphi  (g)\rangle =\left( \frac d{dt}\right) _{t=0}\varphi (e^{tX }g),  \nonumber \\
\langle X, \nabla^{\prime} \varphi (g) \rangle =\left( \frac d{dt}\right)
_{t=0}\varphi (ge^{tX }),~~X \in \g .  \label{grad1}
\end{eqnarray}
If $K\subset G$ is a Lie subgroup, $\k\subset \g$ its Lie algebra and $\overline{\k}\subset \g$ the image of $\k$ in $\g^*\simeq \g$ under the identification $\g\simeq \g^*$ induced by the Killing form then the left and the right gradients of functions on $K$ can be regarded as $C^\infty$-functions on $K$ with values in $\overline{\k}\subset \g$.

The canonical Poisson bracket on the Poisson--Lie group $G$ with
the tangent bialgebra $(\g, \g^*)$ has the form:
\begin{equation}
\{ \varphi ,\psi \}~~=~~\frac 12 \left\langle r \nabla
\varphi,\nabla \psi \right\rangle -~\frac 12\left\langle
r\nabla^{\prime} \varphi , \nabla^{\prime} \psi \right\rangle.
\label{pbr}
\end{equation}

The canonical Poisson bracket on the dual Poisson--Lie group $G^*$
can be described in terms of the original group $G$ and the
classical r--matrix $r$.

We shall use gradients of a function $\varphi\in C^\infty\left( G^*\right) $ with respect to the $G^{*}$ group structure,
\begin{eqnarray}
\langle X, \nabla \varphi  (L_{+},L_{-})\rangle = \left( \frac d{ds}\right) _{s=0}\varphi (e^{sX_{+}}L_{+},e^{sX_{-}}L_{-}), \nonumber \\
\langle X, \nabla^{\prime} \varphi (L_{+},L_{-}) \rangle =\left( \frac d{ds}\right) _{s=0}\varphi  (L_{+}e^{sX_{+}},L_{-}e^{sX_{-}}),X \in \frak g.
\end{eqnarray}

\begin{proposition}
Let $(L_+,L_-)\in G^*$. Then for $f,g \in C^\infty(G^*)$ one has
\begin{equation}\label{Pbracket}
\left\{ f ,g \right\} _{G^*} (L_+,L_-) =
\left\langle (AdL_{+} -AdL_{-})\nabla^{\prime} f
,\nabla g \right\rangle
-\left\langle \nabla f ,(AdL_{+}
-AdL_{-}) \nabla^{\prime} g \right\rangle .
\end{equation}

Denote by $G_*$ the group $G$ equipped with the following Poisson
bracket
\begin{equation}
\left\{ \varphi ,\psi \right\} _* =-\left\langle r \nabla
\varphi,\nabla \psi \right\rangle -\left\langle r \nabla^{\prime
}\varphi,\nabla^{\prime }\psi\right\rangle +2\left\langle r_{-}
\nabla^{\prime }\varphi,\nabla \psi\right\rangle +2\left\langle
r_{+} \nabla\varphi,\nabla^{\prime }\psi\right\rangle ,
\label{tau}
\end{equation}
where all the gradients are taken with respect to the original
group structure on $G$.

Then the map $q:G^* \rightarrow G_*$ defined by
\begin{equation}\label{q*}
q(L_+,L_-)=L_-L_+^{-1}
\end{equation}
is a Poisson mapping and the image of $q$ is a dense open subset
in $G_*$.

There exists a unique left Poisson group action
$$
G\times G^*\rightarrow G^*,~~(g,(L_+,L_-))\mapsto g\circ
(L_+,L_-),
$$
i.e. the action map is Poisson assuming that $G\times G^*$ is equipped with the product Poisson structure, and  if  $q:G^* \rightarrow G_*$ is the map defined by formula
(\ref{q*}) then
$$
q(g\circ (L_+,L_-))=gL_-L_+^{-1}g^{-1}.
$$
\end{proposition}

Now we describe the Poisson structure of Poisson q-W--algebras.
First restrict the action of $G$ on $G_*$ by conjugations to the subgroup $N\subset G$.
Let $s\in G$ be a representative of the element $s\in W$.
By Proposition \ref{prop2} the quotient $NZs^{-1}N/N\simeq s^{-1}ZN_{s^{-1}}$ is a subspace of the quotient $G_*/N$

The following proposition summarizes the results of Section 5 in \cite{S6}.
\begin{proposition}\label{Reduce}
Let $G$ be a complex connected semi--simple algebraic group, $\g$ its Lie algebra, $\h\subset \g$ a Cartan subalgebra in $\g$.
Let $s\in W$ be an element of the Weyl group of the pair $(\g,\h)$, $\p$ the parabolic subalgebra  associated to $s$, $\n$ the nilradical of $\p$, $N\subset G$ the subgroup corresponding to $\n$, $r$ the classical r-matrix (\ref{r}) on $\g$. Equip $G$ with Poisson bracket (\ref{tau}) and denote the corresponding Poisson manifold by $G_*$. Restrict the conjugation action of $G$ on $G_*$ to the subgroup $N$. Then $C^\infty(G_*)^N$ is a Poisson subalgebra in $C^\infty(G_*)$ so that
the Poisson structure on $G_*$ induces a reduced Poisson structure on $G_*/N$, Hamiltonian vector fields of $N$--invariant functions on $G_*$ are tangent to $NZs^{-1}N$ and $N_sZs^{-1}\simeq NZs^{-1}N/N$ is a smooth Poisson submanifold of $G_*/N$.

Let $\pi: NZs^{-1}N \rightarrow N_sZs^{-1}\simeq NZs^{-1}N/N$ be the canonical projection. Then for any (locally defined) smooth functions $\varphi,\psi$ on $N_sZs^{-1}$, and any (locally defined) smooth extensions
$\overline{\varphi},\overline{\psi}$ of $\varphi \circ \pi, \psi \circ \pi$ to $G_*$ we have
\begin{equation}\label{Pbr}
\{\varphi ,\psi \}\circ \pi=\{\overline{\varphi},\overline{\psi}\}_{G_*} \circ i = \left(P_{G_*},d\overline{\varphi}\wedge d\overline{\psi}\right)\circ i,
\end{equation}
where $P_{G_*}$ is the Poisson tensor of $G_*$, $\left( \cdot,\cdot \right)$ is the canonical pairing and $i:NZs^{-1}N \rightarrow G_*$ is the inclusion.
\end{proposition}

Note that for any $X\in \g$ and any regular function $\varphi$ on $G$ the functions $\langle\nabla \varphi,X \rangle$ and  $\langle\nabla ^{\prime }\varphi,X \rangle$ are regular. Therefore the space of regular functions on $G$ is closed with respect to Poisson bracket (\ref{tau}).
Since by Proposition \ref{prop2} the projection $NZs^{-1}N\rightarrow N_sZs^{-1}$ induced by the map $G_*\rightarrow G_*/N$ is a morphism of varieties, Proposition \ref{Reduce} implies that the algebra of regular functions on $N_sZs^{-1}$ is closed under the reduced Poisson bracket defined on $N_sZs^{-1}$ in the previous theorem. This Poisson algebra is called the Poisson q-W--algebra associated to the Weyl group element $s\in W$, or, more precisely, to the conjugacy class of $s\in W$.

The reduced Poisson structure on the slice $N_sZs^{-1}$ is explicitly described in the following proposition.
\begin{proposition}\label{pred1}
Identify $N_sZs^{-1}$ with the subgroup $N_sZ\subset G$ using the right translation by $s$. Let $\n_s$ be the Lie algebra of $N_s$, $\z$ the Lie algebra of $Z$ so that as a linear space the Lie algebra of $N_sZ$ is $\n_s+\z \subset \g$. Then the reduced Poisson bracket on $N_sZs^{-1}\simeq N_sZ$ is given by
\begin{eqnarray}\label{tau1}
\left\{ \varphi ,\psi \right\}(m) =-\left\langle r \nabla
\varphi,\nabla \psi \right\rangle -\left\langle r {\rm Ad}(sm^{-1})\nabla\varphi,{\rm Ad}(sm^{-1})\nabla\psi\right\rangle + \\
+2\left\langle r_{-}
{\rm Ad}(sm^{-1})\nabla\varphi,\nabla \psi\right\rangle +2\left\langle
r_{+} \nabla\varphi,{\rm Ad}(sm^{-1})\nabla\psi\right\rangle , \nonumber
\end{eqnarray}
where $\varphi, \psi\in C^\infty(N_sZ)$, $m\in N_sZ$, and the left gradients $\nabla
\varphi,\nabla \psi$ are regarded as $C^\infty$-functions on $N_sZ$ with values in $\overline{\n_s+\z}\subset \g$.
\end{proposition}

\begin{proof}
We shall use formula (\ref{Pbr}) to prove this proposition. Firstly we use the right trivialization of the tangent and of the cotangent bundle to $G$ and identify them using the Killing form so that differentials of functions on $G$ become their left gradients which can be regarded as functions on $G$ with values in $\g$. Observing that the right gradient of a function $\varphi$ on $G$ is related to its left gradient by
$$
\nabla^{\prime}\varphi (g)={\rm Ad}~g^{-1}(\nabla\varphi(g))
$$
we can rewrite formula (\ref{tau}) for the Poisson bracket on $G_*$ in the form
\begin{eqnarray}\label{tau2}
\left\{ \varphi ,\psi \right\}(g) =-\left\langle r \nabla
\varphi,\nabla \psi \right\rangle -\left\langle r {\rm Ad}(g^{-1})\nabla\varphi,{\rm Ad}(g^{-1})\nabla\psi\right\rangle + \\
+2\left\langle r_{-}
{\rm Ad}(g^{-1})\nabla\varphi,\nabla \psi\right\rangle +2\left\langle
r_{+} \nabla\varphi,{\rm Ad}(g^{-1})\nabla\psi\right\rangle , \nonumber
\end{eqnarray}

Now we calculate the reduced Poisson bracket of two functions $\varphi, \psi\in C^\infty(N_sZ)$ using formulas (\ref{Pbr}) and (\ref{tau2}). We evaluate the left hand side and the right hand side of (\ref{Pbr}) at point $ms^{-1}$, $m\in N_sZ$. Observe that the reduced space $N_sZ\simeq N_sZs^{-1}$ can be regarded as a submanifold of $NZs^{-1}N$ therefore the right hand side of (\ref{Pbr}) evaluated at point $ms^{-1}$ is equal to the reduced Poisson bracket of $\varphi$ and  $\psi$ at point $m$.

Now we justify that the right hand side of (\ref{Pbr}) is equal to the right hand side of (\ref{tau1}).
Denote $\varphi^*=\varphi \circ \pi, \psi^*=\psi \circ \pi$ and let $P_{G_*}(d\overline{\varphi})$ be the Hamiltonian vector field corresponding to the function $\overline{\varphi}$. Without loss of generality we can assume that the extension $\overline{\varphi}$ is a (locally defined) $N$--invariant function on $G_*$. Since by Proposition \ref{Reduce} Hamiltonian vector fields of $N$--invariant functions on $G_*$ are tangent to $NZs^{-1}N$ the right hand side of (\ref{Pbr}) can be rewritten as
$$
\left(P_{G_*},d\overline{\varphi}\wedge d\overline{\psi}\right)\circ i(ms^{-1})=\left(P_{G_*}(d\overline{\varphi}), d\overline{\psi}\right)(ms^{-1})=\left(P_{G_*}(d\overline{\varphi}), d{\psi^*}\right)(ms^{-1}).
$$

Recall that we use the right trivialization of the tangent bundle to $G$ and to $N_sZ$ and identify them with the corresponding cotangent bundles. The first identification also induces identifications of the tangent and of the cotangent bundles to $NZs^{-1}N$ with a subbundle of the tangent bundle to $G$, and of the tangent bundle to $N_sZs^{-1}$ with a subbundle of the tangent bundle to $NZs^{-1}N$.
Since the function $\psi^*$ is the $N$--invariant function on $NZs^{-1}N$ the restriction of which to $N_sZs^{-1}\simeq N_sZ$ is equal to $\psi$, under these identifications $d{\psi^*}$ at the point $ms^{-1}\in N_sZs^{-1}\simeq N_sZ$ is equal to the differential of $\psi$,  so $d{\psi^*}(ms^{-1})=d\psi(m)=\nabla \psi(m)\in \overline{\n_s+\z}\subset \g$ and $P_{G_*}(d\overline{\varphi})(ms^{-1})\in \g$.
This yields
$$
\left(P_{G_*},d\overline{\varphi}\wedge d\overline{\psi}\right)\circ i(ms^{-1})= \left(P_{G_*}(d\overline{\varphi})(ms^{-1}), d\psi(m)\right),
$$
where the pairing in the right hand side of the last formula is induced by the Killing form.

Finally, by the skew--symmetry of the Poisson bracket
\begin{equation}\label{lst}
\left(P_{G_*},d\overline{\varphi}\wedge d\overline{\psi}\right)\circ i(ms^{-1})=\left(P_{G_*}(ms^{-1}),d\varphi(m)\wedge d\psi(m)\right),
\end{equation}
where $d\varphi(m), d\psi(m)\in \overline{\n_s+\z}\subset \g$, $P_{G_*}(ms^{-1})\in \g\wedge \g$ and the pairing in the right hand side of the last formula is induced by the Killing form. Now recalling that $d\varphi(m)=\nabla \varphi(m), d\psi(m)=\nabla \psi(m)$ and using (\ref{tau2}) in the right hand side of formula (\ref{lst}) we arrive at (\ref{tau1}).

\end{proof}

From Propositions \ref{proj} and \ref{Reduce} we can immediately deduce another description of the reduced Poisson bracket on $N_sZs^{-1}$.
\begin{proposition}\label{pred2}
Assume that $G$ is simply connected.
Let $\overline{\varphi},\overline{\psi}$ be two regular functions on $G_*$ the restrictions of which to $N_sZs^{-1}\subset G_*$ are equal to $\varphi,\psi$, respectively.  Then the reduced Poisson bracket $\{\varphi ,\psi \}$ of $\varphi$ and $\psi$ satisfies the following identity
\begin{equation}
\{\varphi ,\psi \}(ms^{-1})=\{\Pi\overline{\varphi},\Pi\overline{\psi}\}_{G_*}(ms^{-1}), m\in N_sZ.
\end{equation}
\end{proposition}

\end{document}